\documentclass[12pt]{article}
\usepackage[utf8]{inputenc}
\usepackage[margin=2cm]{geometry}
\usepackage[dvipsnames]{xcolor}
\usepackage{standalone}
\usepackage{graphicx}
\usepackage{mathrsfs}  
\usepackage{amsfonts,amsmath,amssymb,amsthm}
\usepackage{hyperref,verbatim}
\usepackage{bm,bbm}
\usepackage{bbold}
\usepackage{booktabs}
\usepackage[capitalise]{cleveref}
\usepackage{tabularx}
\usepackage{placeins}
\usepackage{authblk}
\usepackage[backend=biber, maxbibnames=99, giveninits=true, style=numeric-comp, isbn=false, doi=true, url=false]{biblatex}
\title{Metastable opinion dynamics with hidden preferences:\\an Ising model with neutral agents}
\author[1]{Simone Baldassarri}
\author[2]{Vanessa Jacquier}
\author[3]{Alessandro Zocca}

\affil[1]{\small Gran Sasso Science Institute, Viale Francesco Crispi 7, 67100 L’Aquila, Italy}
\affil[2]{\small University of Padova, Via Luigi Luzzatti 4, 35121 Padova, Italy}
\affil[3]{\small Vrije Universiteit Amsterdam, De Boelelaan 1105, 1081 HV Amsterdam, The Netherlands}
\date{}

\newcommand{\s}{\sigma}
\newcommand{\bs}{\mathbf{s}}

\newcommand{\cR}{\mathcal{R}}
\newcommand{\cV}{\mathcal{V}}
\newcommand{\cX}{\mathcal{X}}
\newcommand{\cG}{\mathcal{G}}
\newcommand{\cC}{\mathcal{C}}
 \def \SES {{\hbox{\footnotesize\rm SES}}}

\newtheorem{thm}{Theorem}[section]

\newtheorem{lem}[thm]{Lemma}
\newtheorem{prop}[thm]{Proposition}
\newtheorem{defi}[thm]{Definition}
\newtheorem{conj}[thm]{Conjecture}
\newtheorem{rem}[thm]{Remark}

\numberwithin{equation}{section}

\counterwithin{figure}{section}
\numberwithin{figure}{section}

\newcommand{\ppone}{\bm{+1}}
\newcommand{\mmone}{\bm{-1}}

\begin{document}

\maketitle

\begin{abstract}
We introduce a new Ising-type framework for opinion dynamics that explicitly separates private preferences from publicly expressed binary opinions and naturally incorporates neutral agents. Each individual is endowed with an immutable hidden preference, while public opinions evolve through Metropolis dynamics on a finite graph. This formulation extends classical sociophysical Ising models by capturing the tension between internal conviction, social conformity, and neutrality.
Focusing on highly symmetric grid networks and spatially structured hidden-preference patterns, we analyze the resulting low-temperature dynamics using the pathwise approach to metastability. We provide a complete characterization of stable and metastable configurations, identify the maximal stability level of the energy landscape, and derive sharp asymptotics for hitting and mixing times. A central technical contribution is a new family of isoperimetric inequalities for polyominoes on the torus, which emerge from a geometric representation of opinion clusters and play a key role in determining critical configurations and energy barriers.
Our results provide a quantitative understanding of how spatial heterogeneity in hidden preferences qualitatively reshapes collective opinion transitions and illustrate the power of geometric and probabilistic methods in the study of complex interacting systems.
\end{abstract}

%\medskip
%\noindent
%\textit{MSC Classification:} 82C20; 60J10; 60K35.


\medskip
\noindent
\textit{Acknowledgements:} SB and VJ are grateful for the support of “Gruppo Nazionale per l’Analisi Matematica, la Probabilità e le loro Applicazioni” (GNAMPA-INdAM). The work of SB was also supported by the European Union’s Horizon 2020 research and innovation programme under the Marie Skłodowska-Curie grant agreement no.\ 101034253 while affiliated with Leiden University. 

%\tableofcontents

\section{Introduction}
Collective opinion formation is often shaped by a discrepancy between what individuals privately believe and what they publicly express. In many social, political, and organizational settings, agents may strategically suppress, moderate, or delay the expression of their true preferences due to social pressure, uncertainty, or a lack of strong conviction. As a result, publicly observable opinions may appear binary and polarized, while underlying preferences remain heterogeneous and partially neutral. Understanding how this hidden structure influences collective dynamics remains a central challenge in quantitative models of social interaction.

The Ising model, originally conceived as a paradigm for ferromagnetism in statistical physics \cite{Ising1925}, has evolved into one of the most influential mathematical frameworks for studying collective phenomena in complex systems. In its classical formulation, spins positioned on the vertices of a graph take values in $\{+1,-1\}$ and interact through pairwise couplings that energetically favor local alignment. 

Although first explored on homogeneous structures such as lattices and complete graphs, the model’s conceptual simplicity and analytical tractability have catalyzed its adoption across a wide spectrum of scientific disciplines. In particular, in recent decades, it has emerged as a canonical mathematical model for binary opinion formation in social systems \cite{baldassarri2023ising, li2019binary, mullick2025sociophysics, frahm2019ising, baldassarri2024opinion}.
Interpreting an individual’s opinion as a binary spin, the Ising framework captures two key mechanisms shaping opinion evolution: social conformity and external influence. Local interactions encode peer effects, in which individuals tend to align their expressed opinions with those of their social contacts, while an external magnetic field naturally represents a societal or institutional bias—such as targeted information campaigns, mass media pressure, or political polarization mechanisms. This class of Ising-like opinion models, albeit stylized, provides a rich quantitative understanding of how local interactions and other model parameters can collectively shape global opinion patterns and characterize transient system behavior.

Consider a weighted graph $G=(V,E)$ with a finite number of nodes, modeling individuals, and a weighted adjacency matrix $W =(w_{ij})_{i,j\in V}$, capturing the strength of their pairwise interactions. For two distinct nodes $i,j$, an edge $(i,j) \in E$ exists with weight $w_{ij} >0$, otherwise $w_{ij}=0$. We assume $w_{ii} = 0$ for all $i\in V$, so that the graph $G$ contains no self-loops. Moreover, we normalize the matrix $W$ row-wise so that 
\[
\sum_{j} w_{ij} = 1, \qquad \forall \, j \in V,
\]
which aligns our setting with the normalized Friedkin–Johnsen (FJ) model \cite{friedkin1990social}. Row-normalizing the interaction weights is not essential for the qualitative behavior of the model, but it simplifies both interpretation and analysis.

Every individual $i\in V$ is characterized by
\begin{itemize}
    \item a hidden preference $s_i \in [-1,1]$, and
    \item a public opinion $\s_i \in \{-1,+1\}$. 
\end{itemize}
We stress that in this model, the hidden preference takes any real value in $[-1,1]$, but the public opinion only takes two discrete values, $-1$ and $+1$. Hidden preferences form a vector $\bs=(s_i)_{i\in V} \in \mathbb{R}^{|V|}$, which is assumed to be immutable, while public opinions can change dynamically.

The distinction between hidden preferences and public opinions is intended to model situations in which individuals may not fully reveal their intrinsic beliefs in their observable behavior. While public opinions are constrained to be binary and directly subject to social pressure, hidden preferences represent private inclinations that influence behavior only indirectly, through an ``energetic cost'' for misalignment. Importantly, neutral hidden preferences play a role that goes beyond merely weakening individual bias. Individuals with a neutral hidden preference ($s_i=0$) are fully susceptible to social influence while lacking an intrinsic tendency to favor either opinion, thereby acting as \textit{neutral agents} within the system. As a result, such agents act as mediators between competing opinion domains rather than as passive individuals. In spatially structured settings, collections of such neutral agents can stabilize interfaces between opposing opinions, alter the geometry of emerging opinion clusters, and significantly modify the energy barriers separating macroscopic phases. As we show in this work, these effects give rise to qualitatively new families of stable and metastable configurations that do not appear in classical Ising models with homogeneous or site-dependent external fields.

We define the state space of all public opinion configurations as $\cX := \{-1,+1\}^{|V|}$. To model the interplay between social conformity and the influence of hidden preferences, including neutrality, we introduce the following \textit{energy function} or \textit{Hamiltonian}:
\begin{equation}\label{def:Hamiltonian}
    H(\s) := - \sum_{i \in V} \s_i s_i - 2\alpha \sum_{(i,j) \in E} w_{ij} \s_i \s_j, \qquad \s \in \mathcal{X},
\end{equation}
where $\alpha >0$ is a positive parameter. This Hamiltonian consists of two terms: the first represents the ``inner alignment'', favoring agreement of an individual's public opinion $\s_i$ with their hidden preference $s_i$. The second term promotes ``social alignment'', i.e., agreement of an individual with their neighbors on the graph $G$. 

The parameter $\alpha$ tunes the relative strength of social versus internal alignment. For small values of $\alpha$, public opinions are primarily driven by individual inclination, and heterogeneous hidden-preference patterns strongly influence the energy of configurations. As $\alpha$ increases, social alignment becomes dominant, progressively favoring spatially homogeneous opinion profiles and suppressing the influence of local preference heterogeneity. Our main results later describe how different values of $\alpha$ induce qualitative changes in the energy landscape and the resulting metastable behavior.

The state space $\cX$ and the Hamiltonian structure mirror the classical Ising model on $G$, while encoding a fundamentally different interpretation of individual bias. Indeed, by introducing two positive parameters $h,J>0$ so that $\alpha = J/2h$, we can rewrite the Hamiltonian as
\begin{equation}\label{eq:Hamiltonianspecial}
    H(\s) = -h \sum_{i \in V} \s_i s_i - J \sum_{(i,j) \in E} w_{ij} \s_i \s_j, \qquad \s \in \mathcal{X},
\end{equation}
in which the first term corresponds to a \textit{local} external field of strength $h s_i$ acting on the spin in site $i$. With specific choices of the hidden preference vector $\bs$, we recover various well-studied Ising model variants. In particular,
\begin{itemize}
    \item If $\bs = \mathbf{0}$, we recover the Ising model with a zero external magnetic field;
    \item If $\bs = \mathbf{1}$ (resp.~$\bs = \mathbf{-1}$), we recover the Ising model with a homogeneous external magnetic field favoring the homogeneous configuration $\ppone$ (resp.~$\mmone$).
\end{itemize}
%, which is the configuration $\s\in\cX$ such that $\s_i=+1$ (resp.\ $\s_i=-1$) for any $i\in V$. We will refer to the configurations $\ppone$ and $\mmone$ as \textit{homogeneous configurations}.
Although the Hamiltonian formally resembles that of an Ising model with a site-dependent external field, both the interpretation and resulting dynamics are fundamentally different. In classical Ising models, the external field acts directly on observable spins, whereas here it represents immutable hidden preferences that influence public opinions only in terms of energy. This separation between private inclination and public expression allows for neutral agents and enables social interactions to temporarily override individual preferences, leading to collective phenomena, such as regime-dependent metastability. This feature also distinguishes our approach from linear averaging models such as the Friedkin–Johnsen framework \cite{friedkin1990social}, where private opinions enter directly and continuously into the update rule. In such models, neutrality is encoded quantitatively through susceptibility or initial conditions, and the resulting dynamics are contractive, precluding the emergence of metastable states or energy-barrier–driven transitions.

We consider a discrete-time single-opinion flip dynamics, where at every step only one individual can change its opinion, from $1$ to $-1$, or vice versa. In other words, from a given opinion configuration $\s \in \cX$, a single step of the dynamics can only reach the configurations $\s^{(i)}=(\s^{(i)}_j)_{j \in V}$ for some $i \in V$, where
\begin{equation}\label{eq:spinconf}
    \s^{(i)}_j :=
    \begin{cases}
    \s_j &\text{ if } i\neq j, \\
    -\s_j &\text{ if } i=j.
    \end{cases}
\end{equation}
is the opinion configuration differing from $\s$ only for individual $i$, in which the opinion is flipped.

We consider a stochastic opinion dynamics $(X_t)_{t\in\mathbb{N}}$ on the space $\cX$ of all possible opinion configurations induced by the Hamiltonian $H$. More specifically, we consider Metropolis transition probabilities given by
\begin{equation}\label{eq:glauber}
    P_\beta(\s,\eta) :=q(\s,\eta)e^{-\beta[H(\eta)-H(\s)]_{+}}, \quad \forall \, \s\neq\eta,
\end{equation}
where $[\cdot]_{+}$ denotes the positive part and $\beta$ is a parameter usually referred to as \textit{inverse temperature} in the statistical physics literature. The scalar connectivity matrix $(q(\s,\eta))_{\s,\eta \in \cX}$ appearing in~\eqref{eq:glauber} is independent of $\beta$ and encodes the fact that only single-opinion flips are allowed (cf.~\cref{eq:spinconf}), since for any $\s\neq\eta$, it is defined as
\begin{equation}\label{eq:conn}
    q(\s,\eta) :=
    \begin{cases}
    \frac{1}{|V|} &\text{ if $\exists \ i\in V$ s.t. $\s^{(i)}=\eta$}, \\
    0 &\text{ otherwise}.
    \end{cases}
\end{equation}
We will refer to the Markovian dynamics induced by~\cref{eq:glauber,eq:conn} as \textit{Metropolis dynamics}. The dynamics is \emph{reversible} with respect to the Gibbs measure
\[
    \mu_\beta(\s) := Z_\beta^{-1} \exp(-\beta H(\s)), \qquad \s \in \cX,
\]
where $Z_\beta=\sum_{\s \in \cX} \exp(-\beta H(\s))$ is the normalizing constant also known as partition function.

In this paper, we study this opinion dynamics model with hidden preference in the regime where $\beta$ grows large, which is instrumental in describing a situation where both the tendency to agree with neighbors and the desire to align our public opinion with the hidden one become stronger. As a result, the homogeneous configurations in which all individuals have the same opinion become more likely and at the same time more rigid and harder to change. In such a regime, it is then interesting to identify the minima of the stationary measure $\mu_\beta$ around which the dynamics, study the asymptotic behavior of the mixing time of the opinion dynamics and characterize the asymptotic properties of the hitting times of the underlying Markov process.

From the methodological standpoint, the regime $\beta \to \infty$ allows us to leverage a well-established set of tools from statistical physics, where this regime is known as the low-temperature limit. In particular, we adopt the so-called \textit{pathwise approach}, a dynamical and geometric framework for analyzing hitting times of complex Markov chains focusing on the structure of the \textit{energy landscape} and the most likely transition paths between its configurations \cite{Cassandro1984,Olivieri1995,Olivieri1996,Manzo2004,NZB15}.

The pathwise approach played a key role in the study of metastability and tunneling phenomena in interacting particle systems, particularly in the Ising model. Metastability \cite{cirillo2022metastability, jacquier2025exploring} arises when the stochastic dynamics become trapped for long time scales in \textit{metastable states}, configurations that are local minima of the Hamiltonian but not the global minima, also known as \textit{stable states}. In low-temperature regimes, Glauber or Metropolis dynamics exhibit exponentially long escape times from such metastable states, with transitions typically driven by the formation of critical droplets whose nucleation overcomes an energetic barrier. The pathwise approach provides a rigorous framework for characterizing these rare transitions by analyzing the full ensemble of microscopic trajectories rather than solely relying on aggregate energetic considerations. This method enables the precise identification of \textit{optimal paths} (those minimizing the communication height between metastable and stable states) and produces sharp asymptotics for exit times and transition probabilities (see the monographs \cite{Olivieri2005,Bovier2015}). The pathwise approach was used to study the low--temperature behavior of models with single--spin--flip Glauber dynamics, e.g. \cite{Apollonio2022,Bet2021,Nardi2019,Betnegative,cirillo2023homogeneous,baldassarri2025critical, jacquier2025emergence, jacquier2025estimate, jacquier2024isoperimetric,kim2021metastability,kim2022metastability}, with Kawasaki dynamics, e.g. \cite{baldassarri2023metastability, Baldassarri2021, Baldassarri2022weak,Baldassarri2022strong,baldassarri2023droplet,baldassarriHN,jacquier2024particle,Gaudilliere2009,Bovier2005,Gaudilliere2005,Nardi2005,denHollander2011,denHollander2012,denHollander2000,gois2015}, and with parallel dynamics, e.g. \cite{Cirillo2008,Cirillo2008bis,Bet2021PCA}.

We emphasize that our goal is not to provide a fully realistic model of social networks, but rather to isolate and rigorously analyze the mechanisms through which hidden preferences and neutrality reshape collective opinion dynamics. To this end, we focus on highly symmetric graph topologies and spatially structured hidden-preference patterns, which allow for a precise geometric characterization of the energy landscape and its metastable features. While these assumptions are stylized, they enable sharp asymptotic results and reveal structural effects that are expected to persist, at least qualitatively, beyond the specific setting considered here. Extending the analysis to less regular graphs or disordered preference configurations is an important direction for future work, but lies beyond the scope of the present paper.

The contributions of this work are threefold. 

First, we introduce a new Ising-type model for opinion dynamics that integrates hidden preferences, thereby enriching the classical Ising framework with an additional layer of heterogeneity. This formulation naturally accommodates a third, neutral stance and separates intrinsic preferences from publicly expressed binary opinions, offering a flexible modeling framework for scenarios in which individuals may selectively suppress or amplify personal views. 

Second, we develop a rigorous asymptotic analysis for the resulting Markovian dynamics on highly symmetric and regular network topologies, allowing us to characterize the stable configurations and quantify the asymptotic behavior of the associated hitting times in the low-temperature regime. By coupling the expressiveness of the hidden-preference formulation with the analytical tractability afforded by these structured graphs, our results reveal how specific patterns of hidden preferences shape the global energy landscape and induce qualitatively distinct transition phenomena.

Lastly, we introduce a geometric description of opinion clusters via polyominoes on the toric grid and derive new isoperimetric inequalities for polyominoes winding around the torus. These results are instrumental in identifying critical configurations and energy barriers and may be of independent interest for metastability analyses of lattice spin systems with periodic boundary conditions.

The rest of the paper is organized as follows. Section \ref{sec:defmodel} introduces the model and states the main results. Section \ref{sec:geometry} introduces a specific geometric representation of opinion configurations and related isoperimetric inequalities. In Section \ref{sec:maxstablevel}, we derive the maximal stability level of the resulting energy landscape, a key quantity for the pathwise approach, while in Section \ref{sec:recurrence} we prove the recurrence properties of the dynamics. The proofs of our main theorems are detailed in Section \ref{sec:proofs}. We conclude by summarizing our contributions and outlining several promising directions for future research in Section \ref{sec:conclusions}.

\section{Model description and main results}
\label{sec:defmodel}
In this paper, we study the effect of hidden immutable opinions on the network's transient dynamics and its stationary behavior. To isolate the effects of neutrality and spatial organization of hidden preferences while retaining analytical tractability, we focus on a prototypical setting in which hidden preferences take values in $\{-1,0,1\}$, with $0$ modeling a neutral stance. This choice allows us to clearly distinguish between positively biased, negatively biased, and neutral agents while preserving the essential mechanisms introduced by the general formulation. We emphasize that this discretization is a modeling choice rather than a fundamental restriction, and serves as a first step toward understanding more general preference distributions.

As a result, given a vector of hidden preferences ${\bf s}=(s_i)_{i\in V}$, the node set $V$ is naturally partitioned into three subsets, grouping the nodes with the same hidden preferences. In other words, 
\[
V=A(\bs)\cup B(\bs)\cup C(\bs),
\]
where we introduced three disjoint subsets $A=A(\bs),B =B(\bs),C = C(\bs) \subseteq V$ such that 
\begin{equation}\label{def:s}
s_i=\left\{
\begin{array}{ll}
+1 &\hbox{if } i\in A, \\
-1 &\hbox{if } i\in B, \\
0 &\hbox{if } i\in C.
\end{array}
\right.
\end{equation}
The Hamiltonian in \eqref{def:Hamiltonian} can then be rewritten as
\begin{equation}\label{eq:Hamiltonian_halfway}
H(\s)= - \sum_{i \in A} \s_i + \sum_{i \in B} \s_i - 2 \alpha \sum_{(i,j) \in E} w_{ij} \s_i \s_j, \qquad \s \in \cX.
\end{equation}

To focus our analysis on the impact of spatial displacement of hidden preferences, we assume the underlying graph $G$ is highly symmetric and regular. More specifically, we take $G=(V,E)$ to be the $N\times N$ grid graph with $N$ an even integer and periodic boundary conditions, to which we refer as \textit{toric grid}. The individuals are identified by a pair of coordinates, i.e., $V=\{(i,j) \in \mathbb Z^2 ~|~ 1 \leq i,j \leq N \}$, and each individual has precisely four neighboring individuals (and thus $w_{ij} = 1/4$ if $(i,j) \in E$). The Hamiltonian in \eqref{eq:Hamiltonian_halfway} can then be further rewritten as
\begin{equation}\label{eq:Hamiltonian}
H(\s)= - \sum_{i \in A} \s_i + \sum_{i \in B} \s_i - \dfrac{\alpha}{2} \sum_{(i,j) \in E} \s_i \s_j, \qquad \s \in \cX.
\end{equation}

In our first analysis of this new model, we focus on the case where the hidden preferences are arranged in non-trivial vertical strips on the toric grid and restrict ourselves to a specific interval of interest for the parameter $\alpha$. The assumptions introduced below define a controlled setting in which the interaction among hidden preferences, neutrality, and social alignment can be rigorously analyzed. While the resulting configuration space is stylized, it enables a precise geometric characterization of opinion clusters and their energetic properties, which is essential for the metastability analysis carried out in the subsequent sections.

Specifically, in the rest of the paper, we make the following technical assumptions:
\begin{itemize}
    \item[(A1)] The hidden preferences ${\bf s}=(s_i)_{i\in V}$ are such that 
    \begin{itemize}
        \item[(i)] the regions $A$ and $B$ are two disjoint vertical strips of fixed width $n$ and $m$, respectively, with $3\leq n\leq m$, and 
        \item[(ii)] the subset $C$ is the union of the two disjoint vertical strips, say $C= S_1\cup S_2$, each of fixed width $k$, with $1\leq k<N/2$.
    \end{itemize}
    \item[(A2)] The parameter $\alpha \in \mathbb{N}$ is an integer such that $2\leq \alpha\leq n$ or $\alpha\geq m + 1$. %(in the case $n=m$ this simply means $\alpha>1$, otherwise we are excluding the case $n<\alpha<m$).
   % \item[(iii)] $\sqrt{N+1}<m<\dfrac{N+2}{3}$.
\end{itemize}

Figure \ref{fig:confsigmat1} shows a schematic representation of a toric grid, with a color scheme highlighting the vertical strips arrangement for the hidden preferences as described in Assumption (A1).

\begin{figure}%[!htb]
    \begin{center}
    \includegraphics[scale=0.35]{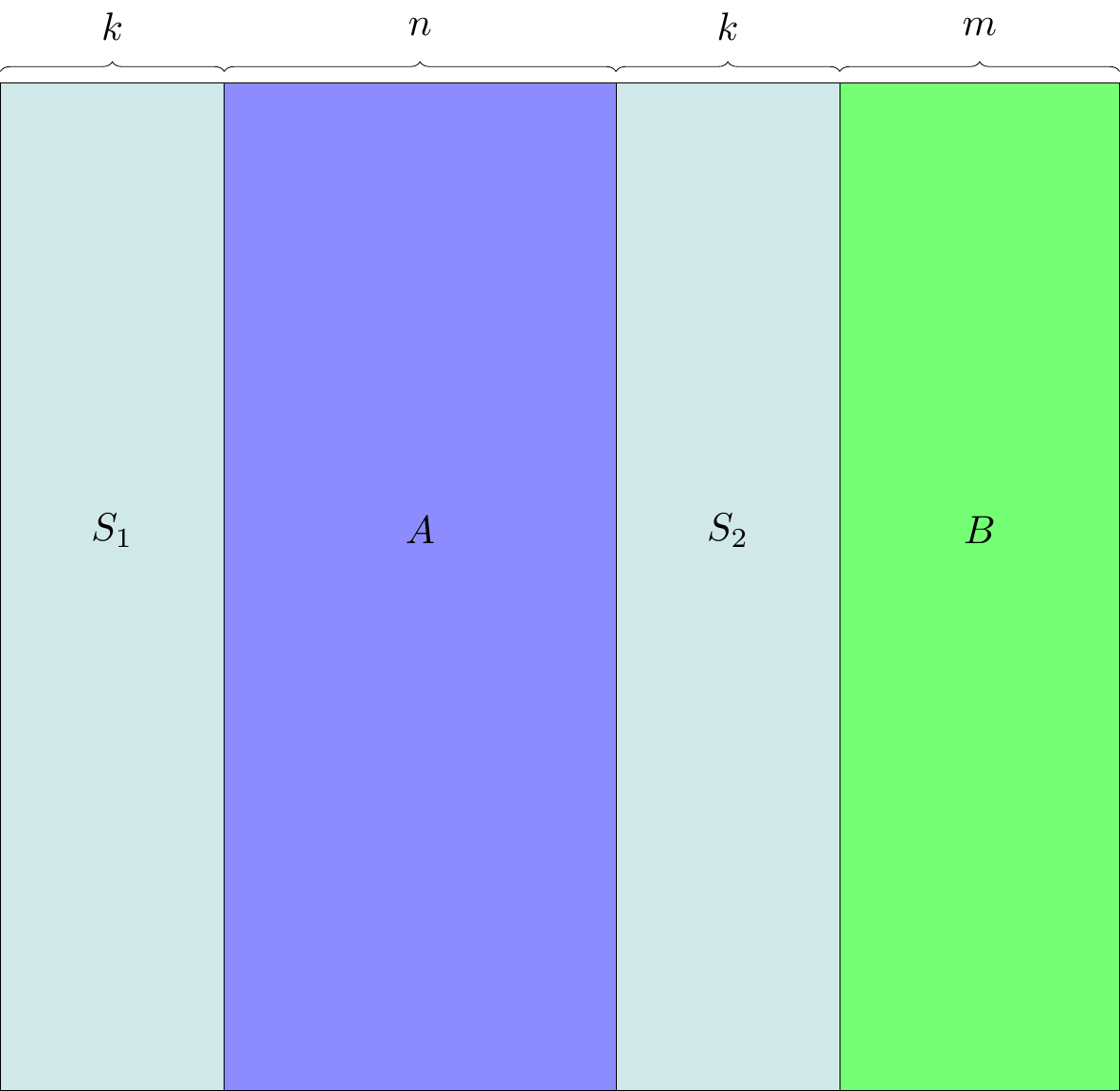}
    \end{center}
    \caption{Schematic representation of the $N\times N$ toric grid $G$ and the subsets $A,B$ and $C=S_1\cup S_2$, where we highlight with different colors (light blue, dark blue and green) the subsets of nodes having hidden preference $0$, $+1$ and $-1$, respectively.}
    \label{fig:confsigmat1}
\end{figure}

Assumption (A1) excludes the less interesting cases corresponding to width $1$ or $2$. Indeed, the number of individuals with non-neutral hidden preferences must be large enough for the system to behave significantly differently compared to the same model with all neutral individuals, i.e., the standard Ising model without an external magnetic field. 
%Having fixed the geometry of the subsets $A$ and $B$, Assumption (A1) implicitly defines also the subset $C$ as the union of the two resulting disconnected components of $V$, which are themselves strips, say $C= S_1\cup S_2$, each of fixed width $1 \leq k<N/2$. 
Lastly, note that Assumption (A1) implies that $N \geq 8$, since $S_1$ and $S_2$ have width $k \geq 1$, $N$ must be even, and $n,m \geq 3$.

Assumption (A2) excludes the parameter region $n<\alpha\leq m$ in which the system behavior is much more involved from a technical point of view, although somewhat tedious in terms of proofs.

\FloatBarrier

\subsection{Main results}
\label{sec:modindep}
This section summarizes the main qualitative and quantitative features of the energy landscape and the induced dynamics. We first identify the global minima of the Hamiltonian, which correspond to the stable configurations of the system. We then characterize the metastable states and determine the maximal stability level that governs low-temperature transitions. Finally, we describe the asymptotic behavior of transition times and identify the critical configurations that any optimal transition path must cross.

Throughout the paper, we will use the following notation. For any $\bar V\subset V$ and any configuration $\s$, we denote by $\s_{|{\bar{V}}}$ the configuration restricted to the subset $\bar V$ and by $\s_{{\bar{V}}}$ the configuration $\s$ such that $\s_{|{\bar{V}}}=\ppone_{|{\bar{V}}}$ and $\s_{|V \setminus {\bar{V}}}=\mmone_{|V \setminus{\bar{V}}}$.

Our first main result identifies the set ${\cal X}^{stab}$ of \textit{stable states} of the system, which are the global minima of the Hamiltonian. 
%Moreover, for any $s_1,s_2\in{\cal X}^{stab}$, we set $\Gamma_s := \Phi(s_1,s_2)-H(s_1)$.
To this end, for $0\leq \ell,p\leq k$, let $\mathscr{A}_{\ell,p}$ be the connected subset of $V$ composed by $A$, $\ell$ adjacent columns in $S_1$ and $p$ adjacent columns in $S_2$ (cf.~Definition \ref{def:sigmat} for more details).

\begin{thm}[Identification of stable configurations]\label{thm:min}
Let $G$ be the $N \times N$ toric graph with $N$ even and assume (A1) holds. Then, the minimum of the Hamiltonian on $\cX$ is equal to
\[
\min_{\s\in\cal X} H(\s)=\left\{
\begin{array}{ll}
-N (n+m) + 2\alpha N - \alpha N^2 &\hbox{if } 2\leq \alpha< n, \\
N(n-m) -\alpha N^2 &\hbox{if } \alpha\geq n.
\end{array}
\right.
\]
The set of stable configurations is
\[
{\cal X}^{stab}=\left\{
\begin{array}{ll}
\displaystyle\bigcup_{\ell,p=0}^k \{\s_{\mathscr{A}_{\ell,p}}\} &\hbox{if } 2\leq \alpha< n, \\
\displaystyle\left\{\mmone,\ppone,\bigcup_{\ell,p=0}^k \{\sigma_{\mathscr{A}_{\ell,p}}\}\right\} &\hbox{if } \alpha = n \hbox{ and } n=m, \\
\left\{\mmone,\displaystyle\bigcup_{\ell,p=0}^k \{\sigma_{\mathscr{A}_{\ell,p}}\}\right\} &\hbox{if } \alpha = n \hbox{ and } n<m, \\
\{\mmone,\ppone\} &\hbox{if } \alpha \geq m+1 \hbox{ and } n=m, \\
\{\mmone\} &\hbox{if } \alpha \geq m+1 \hbox{ and } n<m.
\end{array}
\right.
\]
\end{thm}
This result shows that the structure of the stable set is highly sensitive to the relative strength of the interaction parameter~$\alpha$. 
\begin{itemize}
    \item For values $2 \leq \alpha < n$, the contribution of the interaction term in~\eqref{eq:Hamiltonian} is not sufficiently dominant to suppress the influence of the heterogeneous effect induced by the hidden-preference strips, and consequently, all configurations $\sigma_{\mathscr{A}_{\ell,p}}$ with $0 \le \ell,p \le k$ attain the minimal energy. This regime, therefore, exhibits a comparatively large family of stable configurations.

    \item When $\alpha = n$, the system reaches a critical balance between the local hidden preference terms and the interaction term, and depending on whether $n=m$, this balance either admits both homogeneous configurations $\mmone,\ppone$ as additional minimizers or retains only $\mmone$ together with the mixed-strip configurations $\sigma_{\mathscr{A}_{\ell,p}}$.
    
    \item Finally, for $\alpha \ge m+1$, the interaction term dominates the Hamiltonian, and the only global minima are the homogeneous configurations (either both $\mmone,\ppone$ when $n=m$, or solely $\mmone$ when $n<m$).
\end{itemize}
Thus, as $\alpha$ increases, the system transitions from a regime with a large and geometrically rich collection of stable configurations to one in which the set of stable configurations collapses to a small number of fully ordered phases.\\

%\subsection{Model-independent definitions and notation}

Before we state the other main results, we first need to introduce a series of concepts and notions that are at the core of the \textit{pathwise approach} \cite{Cassandro1984,Manzo2004}, which will be used throughout the paper.

A \textit{path} $\omega$ is a finite sequence $\omega=(\omega_1,\dots,\omega_k)$ of configurations $\omega_i\in{\cal X}$ such that all transition probabilities between consecutive configurations are nonzero, that is $P(\omega_i,\omega_{i+1})>0$ for $i=1,\dots,k-1$.
We write $\omega\colon\;\eta\to\eta'$ to denote a path from $\eta$ to $\eta'$,
namely with $\omega_1=\eta,$ $\omega_k=\eta'$. %A set
%$\cA\subset\cX$ with $|\cA|>1$ is \textit{connected\/} if and only if for all
%$\h,\h'\in\cA$ there exists a path $\o:\h\to\h'$ such that $\o_i\in\cA$
%for all $i$. 
We indicate with $(\omega_1,\omega_2)$ the composition of two paths $\omega_1$ and $\omega_2$. Given the shape of the transition probabilities~\eqref{eq:glauber}, the existence of a path is independent of the inverse temperature $\beta$ and the Hamiltonian $H$, but instead is fully determined by the connectivity matrix $q$, see~\eqref{eq:conn}. If we also consider the Hamiltonian $H$, we can look at the state space $\cX$ as an \textit{energy landscape} and introduce new relevant quantities.

\noindent The communication height between two configurations $\eta,\eta'$ along a path $\omega: \eta\rightarrow\eta'$ is
\begin{equation}\label{eq:commheightpath}
\Phi_{\omega}(\eta,\eta'):= \max_{\zeta\in\omega} H(\zeta).
\end{equation}
and the \textit{communication height} between two configurations $\eta$, $\eta'\in\cal X$ is
\begin{equation}
\Phi(\eta,\eta'):= \min_{\omega:\eta\rightarrow\eta'} \Phi_{\omega}(\eta,\eta') = \min_{\omega:\eta\rightarrow\eta'}\max_{\zeta\in\omega} H(\zeta).
\end{equation}
Similarly, we also define the communication height between two sets $W,W' \subset \mathcal{X}$ as
\begin{equation}
\Phi(W,W'):=\min_{\sigma \in W,\eta \in W'} \Phi(\sigma,\eta).
\end{equation}
The \textit{stability level} of a state $\zeta \in \cal X$ is the energy barrier
\begin{equation}\label{stab}
\mathscr V_{\zeta} :=
\Phi(\zeta,{\cal I}_{\zeta}) - H(\zeta),
\end{equation}
where $\cal I_{\zeta}$ is the set of configurations with energy below $H(\zeta)$, i.e., ${\cal I}_{\zeta}:=\{\eta \in {\cal X}: H(\eta)<H(\zeta)\}$. If $\cal I_\zeta$ is empty, we set $\mathscr V_\zeta:=\infty$.

\noindent Define as \textit{$L$-irreducible configurations} all the configuration $\s \in \cX$ with stability level larger than  $L$, i.e.,
\begin{equation}\label{xv} 
{\cal X}_L:=\{\eta\in{\cal X}: \mathscr V_{\eta}>L\}.
\end{equation}

\noindent The set of \textit{metastable states} is given by
\begin{equation}\label{st.metast.} 
{\cal X}^{meta}:=\{\eta\in{\cal X}:
\mathscr V_{\eta}=\max_{\zeta\in{\cal X}\setminus {\cal X}^{stab}} \mathscr V_{\zeta}\}.
\end{equation}
We denote the \textit{maximum stability level} as $\Gamma_m := \max_{\zeta \in \cX \setminus {\cal X}^{stab}} \mathscr V_{\zeta}$, namely the stability level of the states in ${\cal X}^{meta}$. We note that, by \cite{Cirillo2013}, $\Gamma_m=\Phi(m,s)-H(m)$, where $m\in{\cal X}^{meta}$ and $s\in{\cal X}^{stab}$.

We now turn to the metastable behavior of the dynamics. The next theorem identifies the maximal stability level of the energy landscape and the corresponding metastable configurations. While the explicit expression of the maximal stability level depends on the interaction parameter $\alpha$, the qualitative picture is simple: metastability is governed by the formation of critical opinion clusters whose geometry depends on the relative strength of social alignment and hidden preferences.

To state the following theorem, which characterizes the maximum stability level and identifies the metastable states, we introduce the following positive parameter
\begin{equation}\label{eq:alpha*}
\alpha^* :=\frac{N(m-1)-2m}{m-2}.
\end{equation}

\begin{thm}[Identification of maximum stability level and metastable states]\label{thm:meta}
Let $G$ be the $N \times N$ toric graph with $N$ even and assume (A1) and (A2) hold. Then, the maximum stability level of the energy landscape is equal to
\begin{equation}\label{eq:Gamma*}
\Gamma^*:=\left\{
\begin{array}{ll}
N(m-n) +2\alpha^2 -\alpha(N^2-2) -2 &\hbox{if } 2 \leq \alpha< n, \\
2n^2 -n(N^2-2) -2 &\hbox{if } \alpha = n \hbox{ and } n = m, \\
N(m-n) +2n^2 -n(N^2-2) -2 &\hbox{if } \alpha = n \hbox{ and } n < m, \\
\min\{\Phi_{\omega_1^*},\Phi_{\omega_2^*},\Phi_{\omega_3^*},\Phi_{\omega_4^*}\} &\hbox{if } m + 1 \leq \alpha < 2k + m, \\ 
-N(n+m) -\alpha N^2 +2(2m-1) +2\alpha(N+m-1) &\hbox{if } 2k + m \leq \alpha < \alpha^*, \\
N(m-n) -\alpha N^2 + 2(N+1)(\alpha-1) &\hbox{if } \alpha \geq \alpha^*,
\end{array}
\right.
\end{equation} 
where $\Phi_{\omega_i^*}$, $i=1,...,4$, are the maxima of the energy on four specific paths $\omega_i^*$ defined later in Definition \ref{def:cammini}. Moreover, the set of metastable states is
\[
{\cal X}^{meta}=\left\{
\begin{array}{ll}
 \{\mmone, \ppone\} &\hbox{if } 2\leq \alpha< n\leq m, \\
   \{\ppone\} \, &\hbox{if } n<m \hbox{ and }\alpha \geq m+1 .
\end{array}
\right.
\]
\end{thm}
Theorem~\ref{thm:meta} provides a complete description of the maximal stability level $\Gamma^*$ and of the metastable configurations of the dynamics, revealing how the energy landscape reorganizes as the interaction strength $\alpha$ varies.  
\begin{itemize}
    \item In the regime $2 \leq \alpha < n$, the system retains two metastable states, $\mmone$ and $\ppone$, whose stability levels coincide.
    \item Once $\alpha \ge m+1$, the interaction progressively dominates the Hamiltonian, and the configuration $\ppone$ becomes the unique metastable state.
\end{itemize}
The explicit expression in~\cref{eq:Gamma*} for $\Gamma^*$ in the different parameter regimes, together with the identification of the critical paths attaining the optimal communication height, shows precisely how the metastable landscape collapses as $\alpha$ increases, reducing the system from a symmetric pair of metastable phases to a single, asymmetrically stable one.
In the special cases of either $\alpha = n$ or $\alpha \geq m+1$ with $n=m$, the system exhibits multiple stable states, and therefore our interest is the study of the transition between them.\\

Our third main result describes the asymptotic behavior of the system as $\beta\to\infty$. Specifically, we study the asymptotic behavior of the system through the following quantities:
\begin{itemize}
    \item Given a non-empty set ${\cal A}\subset{\cal X}$ and a state $\sigma\in{\cal X}$, we define the \textit{first-hitting time} of ${\cal A}$ with initial state $\sigma$ at time $t=0$ as
\begin{equation}\label{tempo}
\tau^{\sigma}_{\cal A}:=\min \{t \in \mathbb{N} ~:~ X_t \in {\cal A} \, ~|~ X_0=\sigma\}.
\end{equation}
\item For any $\gamma \in (0,\frac{1}{2})$, we define the \textit{mixing time} as
\begin{equation}
t^{(\beta)}_{mix} = t_{mix}(\gamma):=\min\{t\in \mathbb{N} ~:~ \max_{\sigma\in\cX}||P_\beta^t(\sigma,\cdot)-\mu(\cdot)||_{TV}\leq\gamma\},
\end{equation}
where $P^t(\cdot,\cdot)$ is the $t$-th power of the transition matrix $(P(\sigma,\eta))_{\sigma,\eta\in\cX}$ defined in \eqref{eq:glauber}, and $||\nu-\nu'||_{TV}:=\frac{1}{2}\sum_{\sigma\in\cX}|\nu(\sigma)-\nu'(\sigma)|$ for any two probability distributions $\nu,\nu'$ on $\cX$. 
\item The \textit{spectral gap} of the Markov chain is defined as
\begin{equation}
\rho_\beta:=1-a_\beta^{(2)},
\end{equation}
where $1=a_\beta^{(1)}>a_\beta^{(2)}\geq...\geq a_\beta^{(|\cX|)}\geq-1$ are the eigenvalues of the matrix $(P(\sigma,\eta))_{\sigma,\eta\in\cX}$ defined in \eqref{eq:glauber}.
\end{itemize}

%tunneling time (resp.\ transition time) for the system started at the stable state $s_1$ (resp.\ metastable state $m$) to reach for the first time the other stable state $s_2$ (resp.\ the stable state $s$), which we denote by $\tau_{s_2}^{s_1}$ (resp.\ $\tau_{s}^{m}$). In what follows, we denote by $(\sigma,\eta)$ a pair of configurations such that either $\sigma\in{\cal X}^{meta}$ and $\eta\in{\cal X}^{stab}$, or $\sigma,\eta\in{\cal X}^{stab}$.

%For any $\beta >0$, we consider the discrete-time Markov chain $\{X^{(\beta)}_t\}_{t \in \mathbb{N}}$ on the finite state space $\cX$ with transition probabilities $P_\beta$ given in~\eqref{eq:glauber}. However, we will omit the explicit dependency on the inverse temperature parameter $\beta$ to keep the notation light.

Having identified the metastable structure of the energy landscape, we can characterize the resulting low-temperature dynamics of the opinion process.

\begin{thm} [Asymptotic behavior of the transition time]\label{thm:transitiontime}
Consider two states $\sigma,\eta \in \cX$ such that either:
\begin{itemize}
    \item $\sigma$ is a metastable state and $\eta$ is a stable state, or
    \item $\sigma$ and $\eta$ are both stable states.
\end{itemize}
Then, the following statements hold.
\begin{itemize}
\item[(i)] For any $\epsilon>0$ $\displaystyle\lim_{\beta \to \infty} \mathbb{P}_{\sigma}({\rm e}^{\beta(\Gamma^{*}-\epsilon)}< \tau^\sigma_{\eta}<{\rm e}^{\beta(\Gamma^*+\epsilon)})=1$;
\item[(ii)] $\displaystyle\lim_{\beta \to \infty} \dfrac{1}{\beta}\log \mathbb{E}\tau_{\eta}^\sigma=\Gamma^*$;
\item[(iii)] $\displaystyle\dfrac{\tau_{y}^x}{\mathbb{E}\tau_{\eta}^\sigma}\overset{d}{\to}{\rm Exp}(1)$ as $\beta\to\infty$;
\item[(iv)] there exist two constants $0<c_1<c_2<\infty$ independent of $\beta$ such that for every $\beta>0$ the spectral gap $\rho_{\beta}$ of the Markov chain (see Section \ref{sec:modindep} for the precise definition) satisfies
$$
%\begin{equation}\label{rocompresopca}
c_1e^{-\beta(\Gamma^*+\gamma_1)} \leq \rho_{\beta} \leq c_2e^{-\beta(\Gamma^*-\gamma_2)},
%\end{equation}
$$
where $\gamma_1,\gamma_2$ are functions of $\beta$ that vanish for $\beta\to\infty$.
\end{itemize}
\end{thm}
Theorem 2.3 shows that the transition, mixing, and relaxation time scales of the dynamics are all governed by the same energetic barrier $\Gamma^*$. In particular, the mixing time and the inverse spectral gap grow exponentially with $\beta$ at rate $\Gamma^*$, confirming that metastability dominates the long-time behavior of the system. Specifically, Theorem~\ref{thm:transitiontime}(iv) implies that
\[
\lim_{\beta\to\infty} \dfrac{1}{\beta}\log t^{(\beta)}_{mix}=\Gamma^*=\lim_{\beta\to\infty}-\dfrac{1}{\beta}\log\rho_\beta,
\]
where $t^{(\beta)}_{mix}$ is the mixing time of the Markov chain, which quantifies how long it takes the empirical distribution of the chain to get close to the stationary distribution (see \cref{sec:modindep} for a precise definition).\\

Beyond time scales, it is also of interest to understand the typical configurations through which transitions occur. This leads to the notion of \textit{gates}, which are the critical configurations that any energetically optimal transition path must visit. The last main result provides a geometric description of these critical configurations for transitions between two stable states or between a metastable state and a stable state.

Before the theorem statement, a rigorous definition of a gate is in order. Denote by $(\eta\to\eta')_{opt} $ the \textit{set of optimal paths} as the set of all paths from $\eta$ to $\eta'$ realizing the min-max in $\cal X$, i.e.,
\begin{equation}\label{optpath}
(\eta\to\eta')_{opt}:=\{\omega:\eta\rightarrow\eta' ~:~ \max_{\xi\in\omega} H(\xi)=  \Phi(\eta,\eta') \}.
\end{equation}
The set of \textit{minimal saddles} between $\eta,\eta'\in\cal X$ is defined as
\begin{equation}\label{minsad}
{\cal S}(\eta,\eta'):= \{\zeta\in{\cal X} ~:~ \exists \, \omega\in (\eta\to\eta')_{opt} \text{ s.t. } \zeta \in \omega \hbox{ and } H(\zeta) = \max_{\xi\in\omega} H(\xi)\}.
\end{equation}
Given a pair $\eta,\eta'\in\cal X$, we say that $\cal W\equiv\cal W(\eta,\eta')$ is a \textit{gate} for the transition $\eta\to\eta'$ if $\cal W(\eta,\eta')\subseteq\cal S(\eta,\eta')$ and $\omega\cap\cal W\neq\emptyset$ for all $\omega\in (\eta\rightarrow\eta')_{opt}$. In words, a gate is a subset of $\cal S(\eta,\eta')$ that is visited by all optimal paths.

Next, we need to introduce some families of opinion configurations on the toric grid, beginning with the following two sets:
\[
\cR^A=\bigcup_{a=n-1}^{N-2}\Sigma_{\mathscr{R}_{n,a,1}^A}, \quad \text{ and } \quad
\cR^B=\bigcup_{a=m-1}^{N-2}\Sigma_{V\setminus\mathscr{R}_{m,a,1}^B},
\]
where, with a slight abuse of notation, $\Sigma_{\mathscr{R}_{n,a,1}^A}$ (resp.\ $\Sigma_{V\setminus\mathscr{R}_{m,a,1}^B}$) denotes the set of configurations having positive opinion in the rectangle $\mathscr{R}_{n,a,1}^A$ (resp.\ $V\setminus\mathscr{R}_{m,a,1}^B$) and negative opinion elsewhere, where the set $\mathscr{R}_{a,b,k}^{S}$ with $S\subseteq V$ is precisely introduced in Definition \ref{def:rettangolo} below. Similarly, we introduce
\[
\mathcal{C}^A=\Sigma_{\mathscr{C}_{1,1}^A}, \quad \text{ and } \quad
\mathcal{C}^B=\Sigma_{V\setminus\mathscr{C}_{1,1}^B},
\]
where $\Sigma_{\mathscr{C}_{1,1}^A}$ (resp.\ $\Sigma_{V\setminus\mathscr{C}_{1,1}^B}$) denotes the set of configurations having positive opinion in the set $\mathscr{C}_{1,1}^A$ (resp.\ $V\setminus\mathscr{C}_{1,1}^B$), which are precisely defined in Definition \ref{def:etaTstrips} below. Finally, $\cG^{A}$ (resp.\ $\cG^B$) represents the set of configurations $\Sigma_{\mathscr{R}^A_{\alpha,\alpha-1,1}}$ (resp.\ $\Sigma_{V\setminus\mathscr{R}^B_{\alpha,\alpha-,1}}$).%, where $\alpha$ is the so--called \textit{critical length} and is defined as
%\begin{equation}\label{eq:lungcrit}
%    \ell_c = \alpha.
%\end{equation} 
See Definition \ref{def:Glaubercritico} below for more details.

\begin{thm}[Gate for the transition from (meta)stable state to stable states] \label{thm:selle}
The gates(s) for the transition from (meta)stable state to stable states are as detailed in Table \ref{tab:1}.
\end{thm}

\begin{table}[!h]
\centering
\setlength{\arrayrulewidth}{1pt}
\renewcommand{\arraystretch}{1.8}
\begin{tabularx}{\textwidth}{p{3.5cm}|p{2cm}|p{4cm}|X}
\centering \textbf{$\mathbf{\alpha}$} & \centering\textbf{$n$ and $m$} & 
\centering \textbf{Transition} & 
\qquad \qquad \qquad \textbf{Gates} \\
\hline
\hline
\centering $2\leq\alpha<n$ & \centering $n \leq m$ & \centering $\mmone\to\bigcup_{\ell,p=0}^k \{\sigma_{\mathscr{A}_{\ell,p}}\}$ & \qquad \qquad \qquad \,\,\,\,  $\cG^A$ \\ 
\hline
\centering $2\leq\alpha<n$ & \centering $n \leq m$ & \centering $\ppone\to\bigcup_{\ell,p=0}^k \{\sigma_{\mathscr{A}_{\ell,p}}\}$ & \qquad \qquad \qquad \,\,\,\, $\cG^B$ \\ 
\hline
\centering $\alpha=n$ & \centering $m=n$  & \centering $\mmone\to\ppone$ & \qquad \qquad \,\,\,\, $\cR^A$ and $\cR^B$ \\
\hline
\centering $\alpha=n$ & \centering $m>n$ & \centering $\mmone\to\bigcup_{\ell,p=0}^k \{\sigma_{\mathscr{A}_{\ell,p}}\}$ & 
\qquad \qquad \qquad \,\,\,\, $\cR^A$ \\
\hline
 \centering $m+1\leq \alpha < 2k+m$& \centering $ n \leq m$ & \centering $\ppone\to\mmone$ & \,\,\,\, $(\cG^A \cup \cR^A \cup \cC^A)$ and  $(\cG^B  \cup \cR^B \cup\cC^B)$ \\
\hline
 \centering $2k+m \leq\alpha<\alpha^*$& \centering $m=n$ & \centering $\ppone\to\mmone$ & \qquad \, $\Sigma_{\mathscr{R}_{n,N-2,1}^A}$ and $\Sigma_{V\setminus\mathscr{R}_{n,N-2,1}^B}$ \\
\hline
\centering $2k+m \leq\alpha<\alpha^*$ & \centering $m>n$ & \centering $\ppone\to\mmone$ & \qquad \qquad \,\,\,\,\, $\Sigma_{V\setminus\mathscr{R}_{n,N-2,1}^B}$ \\
\hline
\centering $\alpha\geq \alpha^*$ & \centering $m=n$ & \centering $\ppone\to\mmone$ & \qquad \qquad \,\,\,\,\, $\cC^A$ and $\cC^B$ \\
\hline
\centering $\alpha\geq \alpha^*$ & \centering $m>n$ & \centering $\ppone\to\mmone$ & \qquad \qquad \qquad \,\,\,\, $\cC^B$
\end{tabularx}
\vskip 0.2cm
\caption{Schematic representation of the gates for the relevant transitions.}
\label{tab:1}
\end{table}

\FloatBarrier

\section{Geometrical properties of opinion configurations}
\label{sec:geometry}
To analyze the transient behavior of the opinion dynamics, it is essential to translate the combinatorial structure of configurations into a geometric language that reveals their energetic properties. In~\cref{sub:representation}, we show how clusters of aligned opinions can be represented as polyominoes on the toric grid, thereby allowing the Hamiltonian to be expressed in terms of simple geometric quantities such as their areas and perimeters. This geometric reformulation not only provides an intuitive picture of how opinion clusters grow, shrink, and interact, but also enables the use of isoperimetric inequalities to identify critical configurations and energy barriers. In~\cref{sub:representation} we present a new isoperimetric principle for polyominoes winding around the torus, which is the core technical result that enables our energy landscape analysis in the subsequent sections.

\subsection{Opinion clusters as polyominos}
\label{sub:representation}
To identify the metastable states, it is convenient to rewrite the Hamiltonian function \eqref{eq:Hamiltonian} in terms of the perimeter and area of the regions with positive opinions. This reformulation allows us to associate each cluster of positive opinions with a purely geometric object, a \emph{polyomino}, and to compute the energy barrier of the model by solving an isoperimetric inequality. This standard technique has been widely used for the Ising model (see, e.g., \cite{alonso1996three,arous1996metastability,Bovier2002bis}).

Given the focus and predominance of geometric arguments and alignment with the Ising literature, we will henceforth refer to nodes of the underlying grid graph more generally as \textit{sites} rather than individuals. Given a configuration $\sigma \in \mathcal{X}$, consider the set $C(\sigma)$ defined as the union of the closed unit square of $\mathbb{R}^2$ centered at sites $i\in V$ with the boundary contained in the dual lattice and such that $\sigma(i)=+1$. Given two components of $C(\sigma)$, we say that they are connected if they share at least one edge. Consider the maximal connected components $C_{1}, \ldots, C_{m}$, $m \in \mathbb{N}$, of $C(\sigma)$ as regions of the plane.
If a maximal connected component wraps around the toric grid, it is called a \emph{positive strip}; otherwise, it is called a \emph{positive cluster}.
This construction leads to a bijection that associates with each configuration a collection of its clusters and positive strips. 
Similarly, we define \emph{the negative strips and negative clusters}.
%Likewise, other geometrical objects may be associated to a configuration $\sigma$ by considering the connected components of unit squares centered at the sites of the lattice with spin value minus one. 
%Among these, there could be a connected component which contains two or three lines that wrap around the torus parallel to the coordinate axes of torus. 
%If this is the case, the component is called a \emph{sea of minuses}. 
Given a configuration $\sigma \in \mathcal{X}$, we denote by $\gamma(\sigma)$ its Peierls contour, that is, the boundary of the clusters. 
The Peierls contours live on the dual lattice and are the union of piecewise segments separating spins with opposite sign in $\sigma$, see Figure \ref{fig:contours}. 
Thus, the boundary of each cluster is made of straight lines and corners, which can be \emph{convex corners} or \emph{concave corners} following the usual $\mathbb{R}^2$ definitions. 

\begin{figure}[!ht]
    \begin{center}
    \includegraphics[scale=0.8]{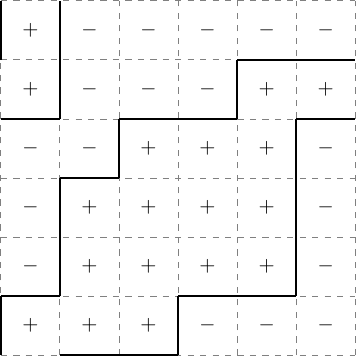}
    \end{center}
    \caption{An example of Peierls contour.}
    \label{fig:contours}
\end{figure}

Given a subset $S \subset V$, let $M_{S}(\sigma)$ be the number of positive opinions in $S$, that is
\[
M_{S}(\sigma) := \sum_{i\in S}\dfrac{\sigma_i+1}{2} = \frac{|S|}{2} + \frac{1}{2} \sum_{i\in S} \sigma_i.
\]
For $p=0,...,N^2$, we introduce the \textit{manifold} $\cV_p$ as the subset of configurations $\sigma \in \cX$ such that $M_V(\sigma)=p$. The space $\cX$ can be partitioned using these manifolds as
\[
\cX = \bigcup_{p=0}^{N^2} \cV_p.
\]
Recall the definition of $A,B\subseteq V$ given in \cref{sec:defmodel} as the subsets describing the hidden preferences. Using the functions $M_A(\cdot)$ and $M_B(\cdot)$, we can rewrite the Hamiltonian \eqref{eq:Hamiltonian} of a configuration $\s$ as follows, making explicit also the total perimeter $|\gamma(\s)|$ of its contours:
\begin{equation}\label{eq:Ham}
\begin{array}{ll}
H(\s)&= |A| - |B| - \alpha N^2 + 2 (M_B(\s)-M_A(\s))  + \alpha |\gamma(\s)| \\
&= N (n-m) - \alpha N^2 + 2 (M_B(\s)-M_A(\s))  + \alpha |\gamma(\s)|.
\end{array}
\end{equation}

%We also formally define rectangles and squares as subsets of $V$ after formally introducing what a polyomino is.

We now introduce new definitions related to the geometry of opinion configurations, which are instrumental for characterizing the energy of certain configurations.

\begin{defi}
    A \emph{polyomino} $P \subset \mathbb{R}^2$ is a finite and maximally edge-connected union of unit squares of $\mathbb{R}^2$. For any subset of sites $S\subset V$, \textit{polyomino associated with} $S$ is the polyomino with the centers of the unit squares in $S\subset V$.
\end{defi}

%We remark that two unit squares are not connected if they share a single point.

\begin{defi}\label{def:rettangolo}
For any $1\leq a,b \leq N-2$, $0\leq k \leq \min\{a,b\}-1$ and $S \subset V$, we denote by $\mathscr{R}^{S}_{a,b,k}\subset V$ the subset of sites in $S$ such that the associated polyomino is an $a\times b$ rectangle with a protuberance of length $k$ attached to the shortest side.
\end{defi}

\begin{defi}\label{def:square} 
For any $1\leq \ell \leq N-2$, $0\leq k \leq \ell-1$ and $S \subset V$, we denote by $\mathscr{Q}^S_{\ell,k}\subset V$ the subset of sites in $S$ such that the associated polyomino is an $\ell\times \ell$ square with a protuberance of length $k$ attached to the longest side.
\end{defi}

Next, we define some peculiar subsets of $V$ that will be used to describe the configurations of special interest, i.e., meta(stable) and critical configurations. To this end, given two sites $i,j\in V$, we denote by $d(i,j)$ the lattice distance between $i$ and $j$. For a subset $\bar V\subseteq V$ and a given vertex $i\in V$, we set
\[
d(i,\bar V) = \min_{j\in \bar V} d(i,j).
\]

\begin{defi}\label{def:sigmat} 
For $0\leq \ell,p\leq k$, we define by $\mathscr{A}_{\ell,p} \subset V$ be the subset of vertices such that 
\begin{align}
    \mathscr{A}_{\ell,p}= A \cup \{ i \in S_1 \, | \, d(i,A) \leq \ell \} \cup \{ i \, \in S_2 \, | \, d(i,A) \leq p \},
\end{align}
\end{defi}
\noindent See Figure \ref{fig:confsigmat} for an illustration of $\mathscr{A}_{\ell,p}$.
\begin{figure}[!ht]
        \begin{center}
        \includegraphics[scale=0.4]{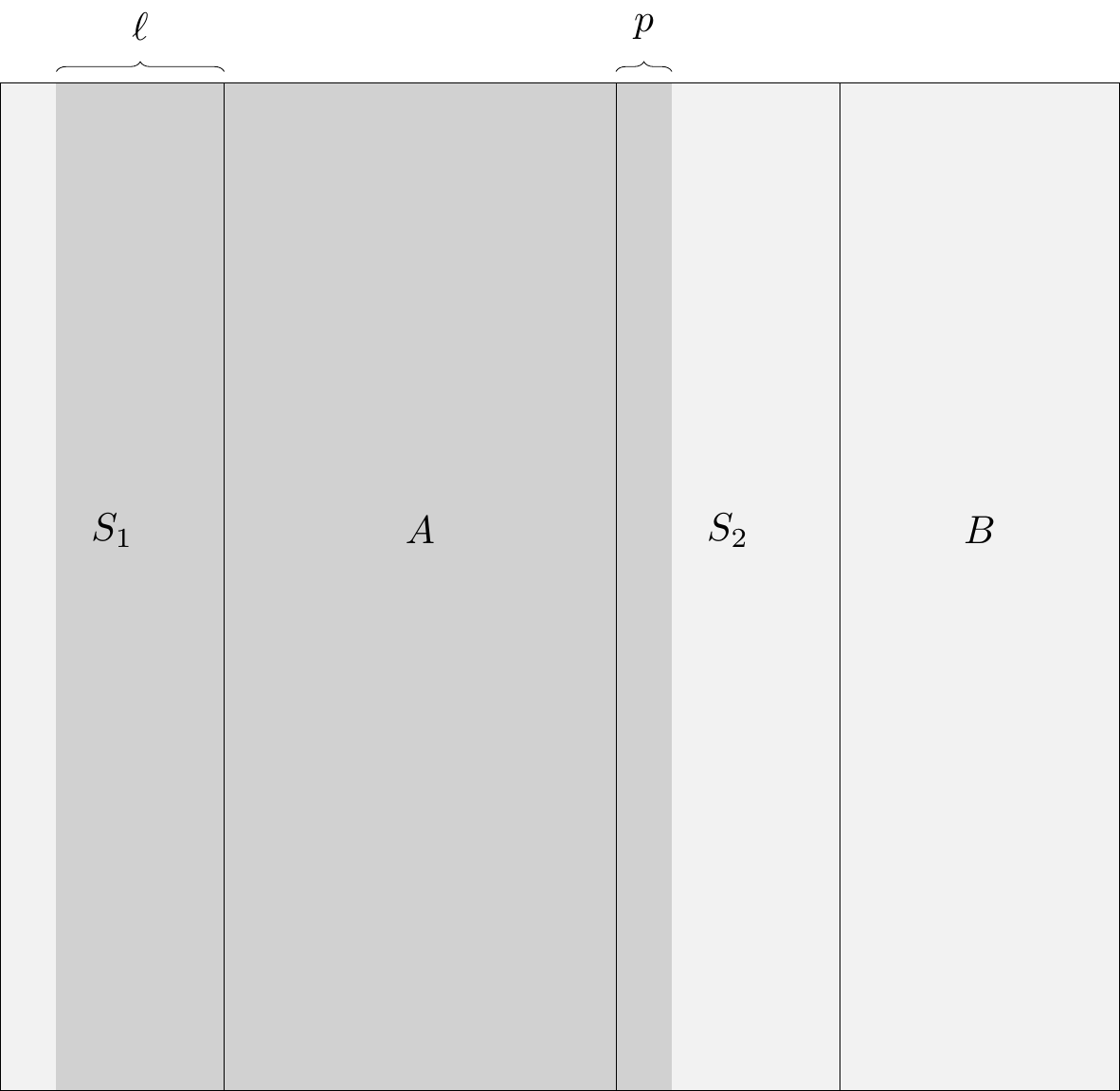}
        \end{center}
        \caption{An example of the subset of sites $\mathscr{A}_{\ell,p}$, which are highlighted in dark grey, with $\ell$ and $p$ adjacent columns in $S_1$ and $S_2$, respectively. %Dark gray denote the sites in $\mathscr{A}_{\ell,p}$.
        }
        \label{fig:confsigmat}
    \end{figure}

\begin{defi}\label{def:etaTstrips} 
For any $r=1,...,n$, $s=1,...,n-1$ and $t=1,...,N$, we define the subset of vertices $\mathscr{C}_{r,s,t}^A\subset V$ as 
\begin{align}
     %\mathscr{C}_{r,s,t}^A= \bigcup_{j=r}^{r+s-1}\{ i \in A \, | \, d(i,S_1) = j \} \cup \mathcal{C}^A_{r+s,t} \cup \mathcal{C}^A_{r-1,t}, 
     \mathscr{C}_{r,s,t}^A= \bigcup_{j=r}^{r+s-1}\{ i \in A \, | \, d(i,S_1) = j \} \cup C^A_{\ell(r,s),t},
\end{align}
where $\ell(r,s) \in \{r+s, r-1\}$, and for any $d=1,...,n$,
\begin{equation}
\begin{array}{ll}
C^A_{d,t}&=\{ (i_j)_{j=1}^t \subset A \, | \, d(i_j,S_1)=d \text{ for } j=1,...,t \text{ and } d(i_j,i_{j+1})=1 \text{ for } j=1,...,t-1 \, \},  %\\
%\widetilde{\mathcal{C}}^A_{s,t}&=\{ (i_j)_{j=1}^t \subset A \, | \, d(i_j,S_2)=s+1 \text{ for } j=1,...,t \text{ and } d(i_j,i_{j+1})=1 \text{ for } j=1,...,t-1 \, \},
\end{array}
\end{equation}
with the convention that 
\[
C^A_{d,1}=\{ \{i_1\} \subset A \, | \, d(i_1,S_1)=d\}, %\qquad \widetilde{\mathcal{C}}^A_{s,t}=\{ \{i_1\} \subset A \, | \, d(i_1,S_2)=s+1\},
\]
and
\[
\mathscr{C}^A_{r,s,0}=\bigcup_{j=r}^{r+s-1}\{ i \in A \, | \, d(i,S_1) = j \}. %, \qquad \widetilde{\mathscr{C}}^A_{s,0}=\{ i \in A \, | \, d(i,S_1) \leq s \}.
\]
We simply refer to $\mathscr{C}_{s,t}^A$ when it does not matter the precise value of $r\in\{1,...,n\}$, see Figure \ref{fig:saddle}. 
\end{defi}

In other words, $\mathscr{C}_{r,s,t}^A$ is the set of vertices in $A$ forming $s$ adjacent columns, with the first one at distance $r$ from the set $A$, and an incomplete column of length $t$, i.e., $t$ sites attached either to the first or to the last column.
Note that 
\[
\bigcup_{j=r}^{r+s-1}\{ i \in A \, | \, d(i,S_1) = j \} = \bigcup_{j=r}^{r+s-1}\{ i \in A \, | \, d(i,S_2) = n-j+1 \}.
\]

We define in a similar way the set of vertices in $B$ forming $s$ adjacent columns, with the first one at distance $r$ from the set $B$, and an incomplete column of length $t$, which we refer to as $\mathscr{C}_{r,s,t}^B$.

% \begin{defi}\label{def:etaTstrips2} 
% For any $r=1,...,m$, $s=1,...,m-1$ and $t=1,...,N$, we define the subset of vertices $\mathscr{C}_{s,t}^B \subset V$ as 
% \begin{align}
%      \mathscr{C}_{r,s,t}^B= \bigcup_{j=r}^{r+s-1}\{ i \in B \, | \, d(i,S_1) = j \} \cup C^B_{\ell(r,s),t}, %\quad \widetilde{\mathscr{C}}_{r,s,t}^A= \bigcup_{j=r}^{r+s-1} \{ i \in A \, | \, d(i,S_2) = j \} \cup \widetilde{\mathcal{C}}^A_{s,t},
% \end{align}
% where $\ell(r,s) \in \{r+s,r-1\}$, and for any $d=1,...,m$,
% \begin{equation}
% \begin{array}{ll}
% C^B_{d,t}&=\{ (i_j)_{j=1}^t \subset B \, | \, d(i_j,S_1)=d \text{ for } j=1,...,t \text{ and } d(i_j,i_{j+1})=1 \text{ for } j=1,...,t-1 \, \},  %\\
% %\widetilde{\mathcal{C}}^A_{s,t}&=\{ (i_j)_{j=1}^t \subset A \, | \, d(i_j,S_2)=s+1 \text{ for } j=1,...,t \text{ and } d(i_j,i_{j+1})=1 \text{ for } j=1,...,t-1 \, \},
% \end{array}
% \end{equation}
% with the convention that 
% \[
% C^B_{d,1}=\{ \{i_1\} \subset B \, | \, d(i_1,S_1)=d\}, %\qquad \widetilde{\mathcal{C}}^A_{s,t}=\{ \{i_1\} \subset A \, | \, d(i_1,S_2)=s+1\},
% \]
% and
% \[
% \mathscr{C}^B_{r,s,0}=\bigcup_{j=r}^{r+s-1}\{ i \in B \, | \, d(i,S_1) = j \}. %, \qquad \widetilde{\mathscr{C}}^A_{s,0}=\{ i \in A \, | \, d(i,S_1) \leq s \}.
% \]
% We simply refer to $\mathscr{C}_{s,t}^B$ when it does not matter the precise value of $r\in\{1,...,m\}$, which identifies the first column of the set following the lexicographic order. 
% \end{defi}

\begin{defi}\label{def:Glaubercritico}
We denote by $\cG^{A}_{prot}$ (resp.\ $\cG^{B}_{prot}$) the set of configurations $\Sigma_{\mathscr{R}^A_{\alpha,\alpha-1,0}}$ (resp.\ $\Sigma_{V\setminus\mathscr{R}^B_{\alpha,\alpha-1,0}}$) and by $\cG^{A}$ (resp.\ $\cG^B$) the set of configurations $\Sigma_{\mathscr{R}^A_{\alpha,\alpha-1,1}}$ (resp.\ $\Sigma_{V\setminus\mathscr{R}^B_{\alpha,\alpha-,1}}$).
\end{defi}

The sets of configurations introduced in Definition \ref{def:Glaubercritico} are both well defined when $\alpha \leq n$. In the case $n<\alpha\leq m$, only the sets $\cG^{B}_{prot}$ and $\cG^{B}$ are well defined.

We conclude this section by presenting three useful lemmas on the energetic properties of specific configurations, which will be used throughout the rest of the paper. 

\begin{lem}[Energy of $\sigma_{\mathscr{A}_{\ell,p}}$]\label{lem:energiatruth}
For any $\ell,p=0,...,k$, 
\[
H(\sigma_{\mathscr{A}_{\ell,p}})= -N (n+m) + 2\alpha N  - \alpha N^2.
\]
\end{lem}
\begin{proof}
Thanks to \eqref{eq:Ham}, it is enough to prove that $M_A(\sigma_{\mathscr{A}_{\ell,p}})=nN$, $M_B(\sigma_{\mathscr{A}_{\ell,p}})=0$ and $|\gamma(\sigma_{\mathscr{A}_{\ell,p}})|=2N$ for any $\ell,p=0,...,k$. This follows directly from Definition \ref{def:sigmat}.
\end{proof}

\begin{lem}[Energy of column configurations]\label{lem:colonna1}
We denote by $\mathbb{1}_{\{\cdot\}}$ the indicator function.
\begin{itemize}
\item[(i)] For any $s=0,...,n-1$ and $t=1,...,N$,
    \begin{align}\label{eq:diffenA}
        H(\Sigma_{\mathscr{C}^A_{s,t}})-H(\s_{\mathscr{A}_{0,0}})%=H(\sigma_{\mathscr{C}^B_{s,t}})-H(\s_{\mathscr{A}_{0,0}})
        =\left\{
        \begin{array}{ll}
    2(N(n-s)-t)+2\alpha \mathbb{1}_{\{t \neq N\}} &\hbox{if } s\neq0, \\
     2(nN-t) + 2\alpha (t-N+1) \mathbb{1}_{\{t \neq N\}} &\hbox{if } s=0.
        \end{array}
        \right.
    \end{align}
\item[(ii)] For any $s=0,...,m-1$ and $t=1,...,N$,
    \begin{align}\label{eq:diffenB}
       % H(\sigma_{\mathscr{C}_{s,t}^A})-H(\s_{\mathscr{A}_{k,k}})=
       H(\Sigma_{V\setminus\mathscr{C}_{s,t}^B})-H(\s_{\mathscr{A}_{k,k}})=\left\{
        \begin{array}{ll}
    2(N(m-s)-t)+2\alpha \mathbb{1}_{\{t \neq N\}} &\hbox{if } s\neq0, \\
     2(mN-t) + 2\alpha (t-N+1) \mathbb{1}_{\{t \neq N\}} &\hbox{if } s=0.
        \end{array}
        \right.
    \end{align}
\end{itemize}
\end{lem}

\begin{proof}
It is enough to prove point (i). Indeed, point (ii) follows by symmetry. %By \eqref{eq:Ham}, it is immediate to see that $H(\sigma_{\mathscr{C}^A_{s,t}})=H(\sigma_{\mathscr{C}^B_{s,t}})$. Moreover,
We deduce that
\begin{align}
    H(\Sigma_{\mathscr{C}^A_{s,t}})-H(\sigma_{\mathscr{A}_{0,0}})=& 2 (M_B(\Sigma_{\mathscr{C}^A_{s,t}})-M_A(\Sigma_{\mathscr{C}^A_{s,t}}))  + \alpha |\gamma(\Sigma_{\mathscr{C}^A_{s,t}})| \notag \\
    & -2 (M_B(\sigma_{\mathscr{A}_{0,0}})-M_A(\sigma_{\mathscr{A}_{0,0}})  -\alpha |\gamma(\sigma_{\mathscr{A}_{0,0}})|.
\end{align}
From Definitions \ref{def:sigmat} and \ref{def:etaTstrips}, we have $M_B(\Sigma_{\mathscr{C}^A_{s,t}})=M_B(\sigma_{\mathscr{A}_{0,0}})$, $M_A(\Sigma_{\mathscr{C}^A_{s,t}})=sN+t$ and $M_A(\sigma_{\mathscr{A}_{0,0}})=nN$. We conclude by quantifying the difference between the perimeter of the cluster of negative opinions in $\Sigma_{\mathscr{C}^A_{s,t}}$ and that of the cluster of negative opinions in $\sigma_{\mathscr{A}_{0,0}}$. First, we consider the case $s=0$. This difference is $2(t-N+1)$ if $t\neq N$ and $0$ if $t=N$. Finally, we consider the case $s\neq0$. This difference is then $2$ if $t\neq N$ and 0 if $t=N$.
\end{proof}

\begin{figure}%[!htb]
        \begin{center}
        \includegraphics[scale=0.45]{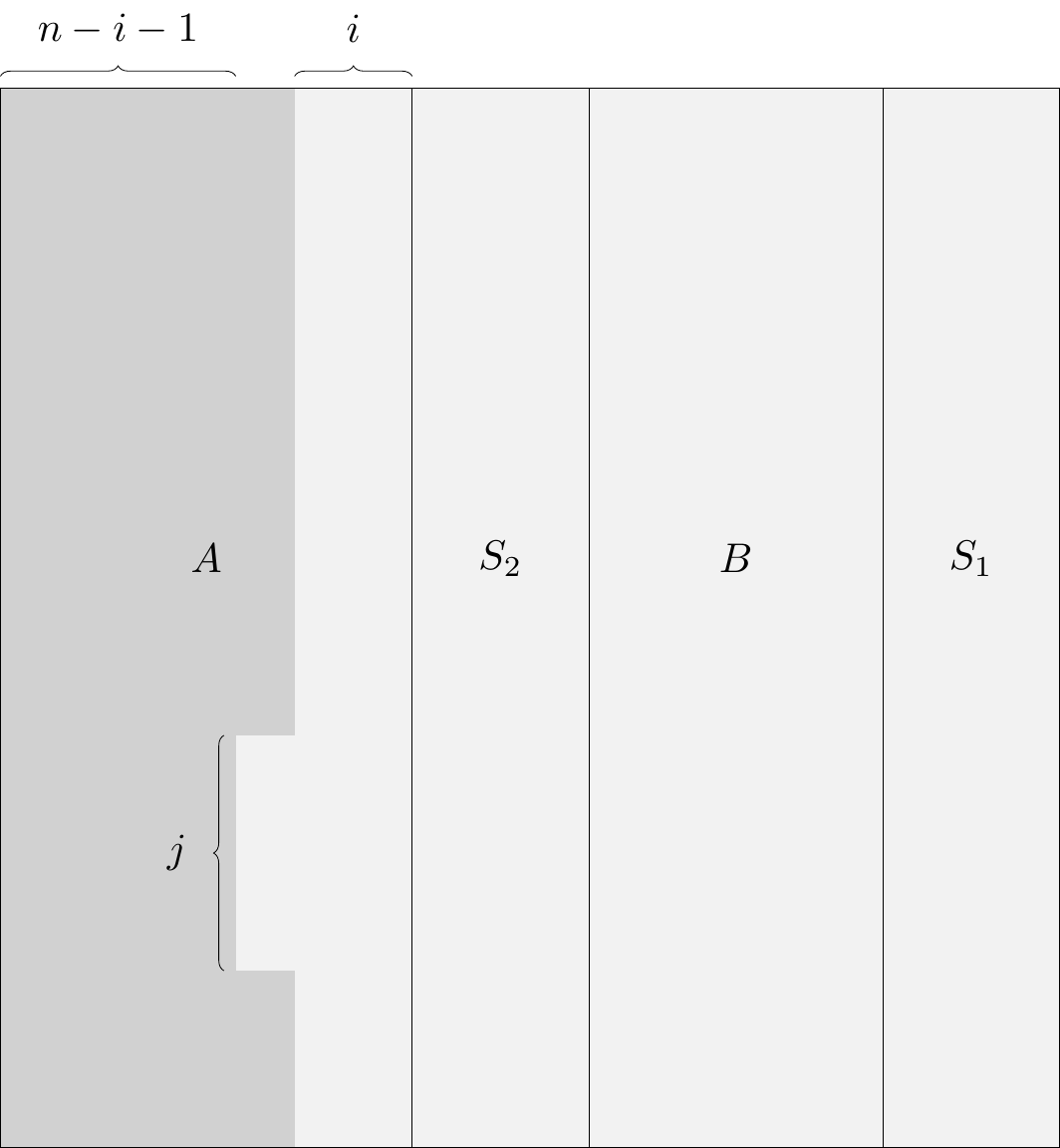}
        \end{center}
        \caption{Representation of a configuration $\Sigma_{\mathscr{C}^A_{n-i-1,N-j}}$. Light and dark gray denote the negative and the positive region, respectively.}
        \label{fig:saddle}
    \end{figure}

    \begin{lem}[Maximal energy of column configurations]\label{lem:enmin}
    For any $\alpha > 2$,
    \begin{itemize}
       \item[(i)] If $m=n$, $\max \{ H(\Sigma_{\mathscr{C}_{0,N-1}^A}), H(\Sigma_{V\setminus\mathscr{C}_{0,N-1}^B}),H(\Sigma_{\mathscr{C}_{1,1}^A}),H(\Sigma_{V\setminus\mathscr{C}_{1,1}^B})\}=H(\Sigma_{\mathscr{C}_{1,1}^A})=H(\Sigma_{V\setminus\mathscr{C}_{1,1}^B})$.
    \item[(ii)] If $m>n$, $\max \{ H(\Sigma_{\mathscr{C}_{0,N-1}^A}), H(\Sigma_{V\setminus\mathscr{C}_{0,N-1}^B}),H(\Sigma_{\mathscr{C}_{1,1}^A}),H(\Sigma_{V\setminus\mathscr{C}_{1,1}^B})\}=H(\Sigma_{V\setminus\mathscr{C}_{1,1}^B})$.
    \end{itemize}
    
\end{lem}
\begin{proof}
    By \eqref{eq:Ham}, we deduce that
    \[
    \begin{array}{ll}
  H(\Sigma_{\mathscr{C}_{0,N-1}^A}) &= N(n-m) - \alpha N^2 -2(N-1) + 2\alpha N, \\
  H(\Sigma_{V\setminus\mathscr{C}_{0,N-1}^B}) &= N(n-m) - \alpha N^2 +2N(m-n) -2(N-1)+ 2\alpha N, \\
  H(\Sigma_{\mathscr{C}_{1,1}^A}) &= N(n-m) - \alpha N^2 -2(N+1)+ 2\alpha (N+1), \\
  H(\Sigma_{V\setminus\mathscr{C}_{1,1}^B}) &= N(n-m) - \alpha N^2 +2N(m-n) -2(N+1)+ 2\alpha (N+1).
    \end{array}
    \]
    Thus, a direct comparison proves the statement.
\end{proof}

\subsection{Isoperimetric inequality for the polyominoes on the toric grid}
\label{sub:isoperimetric}
We characterize spin configurations that are relevant for the dynamics under consideration in terms of the area and the perimeter of the polyominoes associated with their clusters. 
Thus, we consider the unit squares of $\mathbb{R}^2$ with the sides parallel to the axes of $\mathbb{R}^2$ and the center in $V$, and we give some characteristics of polyominoes. Some of the following definitions are rather standard when dealing with isoperimetric inequalities (see, e.g., \cite{alonso1996three,Cirillo2013}), but we report them below to make the presentation self-contained.

\begin{defi}\label{def:perimeter}
 	Given a polyomino $P$, its perimeter $p(P)$ is the number of unit squares belonging to $\mathbb{T}^2 \setminus P$ and sharing at least an edge with the polyomino.
\end{defi}

\begin{defi}
    A \emph{quasi-square} is a polyomino having rectangular shape $\ell_1\times\ell_2$, with $2\leq\ell_1\leq\ell_2\leq\ell_1+1$.
    A \emph{$k$-strip} is a polyomino such that there exist $k\in\mathbb{N}$ lines that wrap around the toric grid parallel to the vertical coordinate axis of $\mathbb{T}^2$. We simply refer to \emph{strip} if we do not emphasize its width.
\end{defi}

\begin{defi}
    A \emph{$q$-protuberance} is an edge-connected subset of $q$ unit squares. We refer to $q$ as the length of the protuberance.
\end{defi}

\begin{defi}
Let $P$ be a polyomino.
We call \emph{hole} of $P$ the union of removed connected unit squares in $P \setminus \partial^- P$, see Figure \ref{fig:convex_hole}, and its cardinality is the number of the unit squares composing it. In other words, the hole of $P$ is a union of connected squares completely surrounded by the unit squares of $P$.
\end{defi}

\begin{defi}\label{def:concave_convex} 
We consider the intersections between a polyomino $P$ and the lines parallel to one of the coordinate axes and wrapping around the toric grid. 
If every intersection is composed by only one connected component of unit squares of $P$, then we call $P$ \emph{convex polyomino}, see Figure \ref{fig:convex_hole}.
A non-convex polyomino is called \emph{concave polyomino}. %$P$ as a polyomino with two concave corners along the same column or row. 
\end{defi}

We note that a polyomino with a hole is a concave polyomino.

\begin{figure}
\begin{center}
    \includegraphics[scale=0.7]{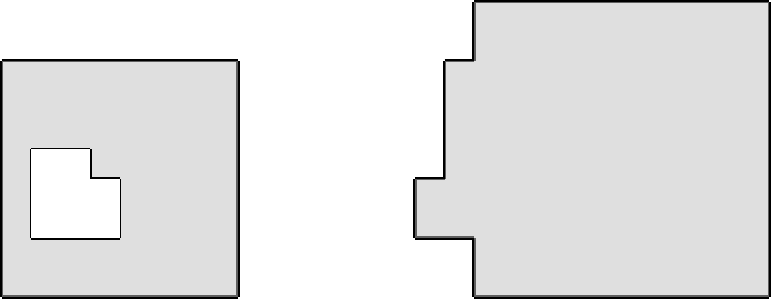}
    \caption{On the left, a polyomino with a hole. On the right, a convex polyomino.}
 \label{fig:convex_hole}
\end{center}
    \end{figure}

\begin{defi}
Let $P$ be a concave polyomino. We consider the intersection between $P$ and one of the lines parallel to one of the coordinate axes and wrapping around the toric grid such that it is composed by at least two connected components of $P$. 
The union of the connected unit squares in $\partial^+ P$ intersected by this line is called \emph{concavity} of $P$,
%the union of the connected unit squares in $\partial^+ P$ between the two concave corners, 
see Figure \ref{fig:concavity}. The number of these unit squares is called the \emph{cardinality} of the concavity. 
The \emph{width} of the concavity is the number of consecutive lines parallel to one of the coordinate axes and wrapping around the toric grid that intersect $P$ in at least two connected components of $P$, see Figure \ref{fig:concavity}.
% is the number of columns (rows, respectively) between the column (row, respectively) containing the two concave corners and the column (row, respectively) containing the nearest convex corner of $P$, see Figure \ref{...}.
\end{defi}
\begin{figure}
\begin{center}
    \includegraphics[scale=0.5]{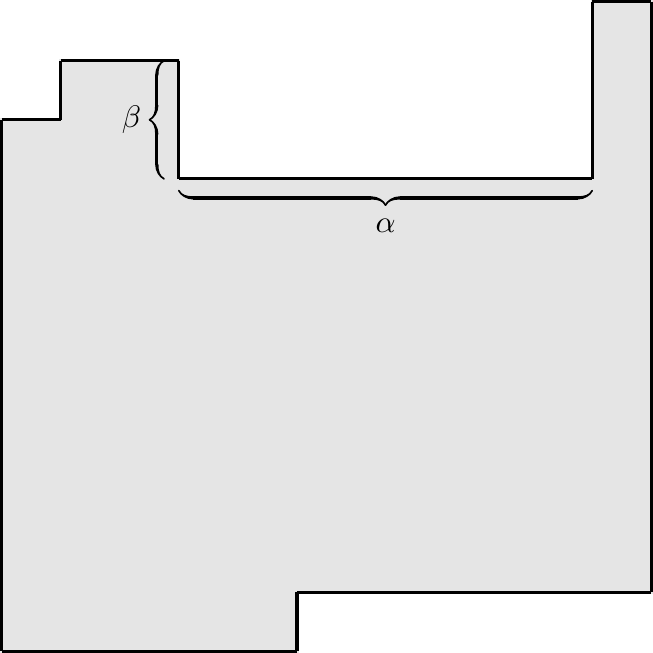}
    \caption{An example of a concave polyomino. Its concavity has cardinality $\alpha$ and width $\beta$.}
 \label{fig:concavity}
\end{center}
    \end{figure}
Note that with this construction, there is a bijection between positive clusters and polyominoes. 

The notions of concavities, protuberances, and holes are introduced to quantify local geometric irregularities of opinion clusters. From an energetic perspective, such irregularities typically increase the perimeter without significantly increasing the area, thereby raising the energy of a configuration. Controlling and eliminating these features is therefore essential for identifying critical configurations that minimize energy barriers along transition paths.

In \cite{Cirillo2013}, the authors proved that, for a fixed area, the polyominoes with minimal perimeter is the so-called quasi-square. Here we will show that, for a fixed but large enough area, the shape of a subset of $\mathbb{R}^2$ with minimal perimeter and winding around the toric grid is a quasi-strip. Moreover, we will prove that the energy of the critical strip is lower than that of the critical quasi-square when the parameter $\alpha$ is large enough. 

\begin{prop}[Isoperimetric inequality for polyominoes on the toric grid]\label{lem:isoperimetrica}
Let $n \in \mathbb{N}$, then
\begin{itemize}
\item[(i)] the set of polyominoes winding around the toric grid with area $n\geq N$ and minimal perimeter are strips with possibly one protuberance attached to a side.
\item[(ii)] the set of polyominoes not winding around the toric grid with area $n$ and minimal perimeter are quasi-squares with possibly one protuberance attached to a side.
    \end{itemize}
\end{prop}

\begin{proof}
We need to consider only case (i), because case (ii) is simply \cite[Lemma 6.17]{Cirillo2013}. 

We consider all the polyominos winding around the toric grid with area $n\geq N$. Suppose by contradiction that such a polyomino $P$ is not a strip with possibly one protuberance, but it has minimal perimeter.

If $P$ has a $1\times k$ rectangular hole, with $k\geq1$, then the polyomino $P'$ obtained after filling the hole with $k$ unit squares removed from $\partial^- P$ is such that $p(P')<p(P)$, which is a contradiction. Consider now the case in which the hole has not a rectangular shape $1\times k$. Let $\mathcal{R}$ be the smallest rectangle circumscribing the hole. We consider $\ell_1,...,\ell_m$ the connected components of the hole that intersect $\partial^-\cR$ and we denote by $\ell_i$ one of the smallest. 
We consider the polyomino $P'$ obtained after filling the hole with $|\ell_i|$ unit squares removed from $\partial^- P$ is such that $p(P') \leq p(P)$ by construction. In the case $p(P')=p(P)$, denote by $\ell_j$ one of the smallest connected components of the hole in $P'$ that intersects $\partial^-\cR$. 
If $|\ell_i|+|\ell_j|<N$, we construct the polyomino $P''$ by filling the hole with $|\ell_j|$ unit squares removed from $\partial^- P'$. Note that $p(P'') < p(P')$ by construction. Indeed, after moving $|\ell_i|$ unit squares from $P$ in the construction of $P'$, two concave corners are created on $\partial^-P$. Then, during the construction of $P''$, the unit squares are moved starting from the corner, so that the perimeter is constant during the first $|\ell_j|-1$ moves, but decreases during the last step. 
Otherwise if $|\ell_i|+|\ell_j| \geq N$, we construct the polyomino $P''$ by filling the hole with $N-|\ell_i|\leq |\ell_j|$ unit squares removed from $\partial^- P'$ starting from one of the corners. Note that $p(P'') < p(P')$ by construction. Indeed, during the construction of $P''$, the unit squares are moved starting from the corner, so that the perimeter is constant during the first $N-|\ell_i|-1$ moves, but decreases during the last step. This is a contradiction.

It remains to consider the case with only one concavity, because when there is more than one concavity, we can argue in the same way. We distinguish two cases:
\begin{itemize}
    \item[(i)] The width of the concavity is one.
    \item[(ii)] The width of the concavity is at least two.
\end{itemize}

{\bf Case (i).} If there is more than one concavity, we consider the polyomino $P'$ after attaching together two protuberances formed by the concavities, see Figure \ref{fig:concavity_movement}. Thus, $p(P')<p(P)$ and we are done. We are remaining with the case in which there is only one concavity with width one. However, in this case, the polyomino is a strip with one protuberance, which contradicts our initial assumption.

\begin{figure}[!h]
\begin{center}
    \includegraphics[scale=0.5]{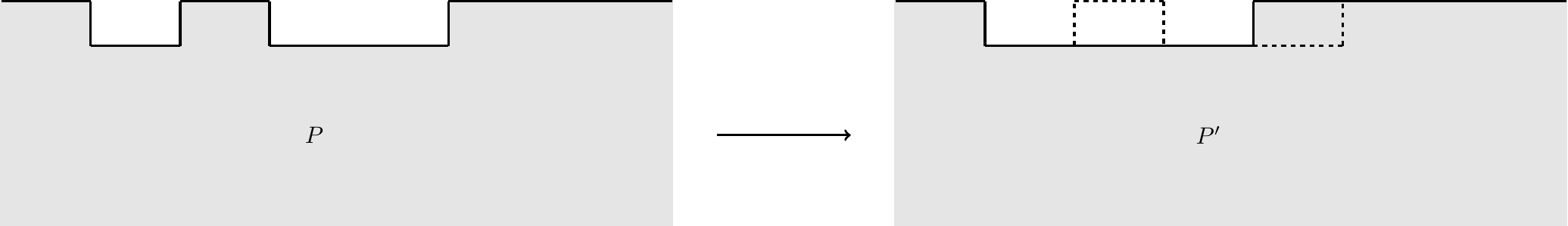}
    \caption{On the left, an example of a polyomino $P$ with two concavities. %We recall that the polyomino wraps around the torus, so the concavities in this picture are two. 
    On the right, the polyomino $P'$ obtained from $P$ by shifting one of the two protuberances.}
 \label{fig:concavity_movement}
\end{center}
    \end{figure}

{\bf Case (ii).} Consider one of the two protuberances formed by the concavity of cardinality $n$, with $n\geq1$, and let $q\geq1$ be the length of such a protuberance. Let $m=\min \{q, n \}$, we define the polyomino $P'$ obtained from $P$ by shifting the $m$ unit squares composing the protuberance to fill (part of) the concavity. See Figure \ref{fig:concavity_movement2}. Thus, $p(P')<p(P)$, which is a contradiction. 

\begin{figure}[!h]
\begin{center}
    \includegraphics[scale=0.38]{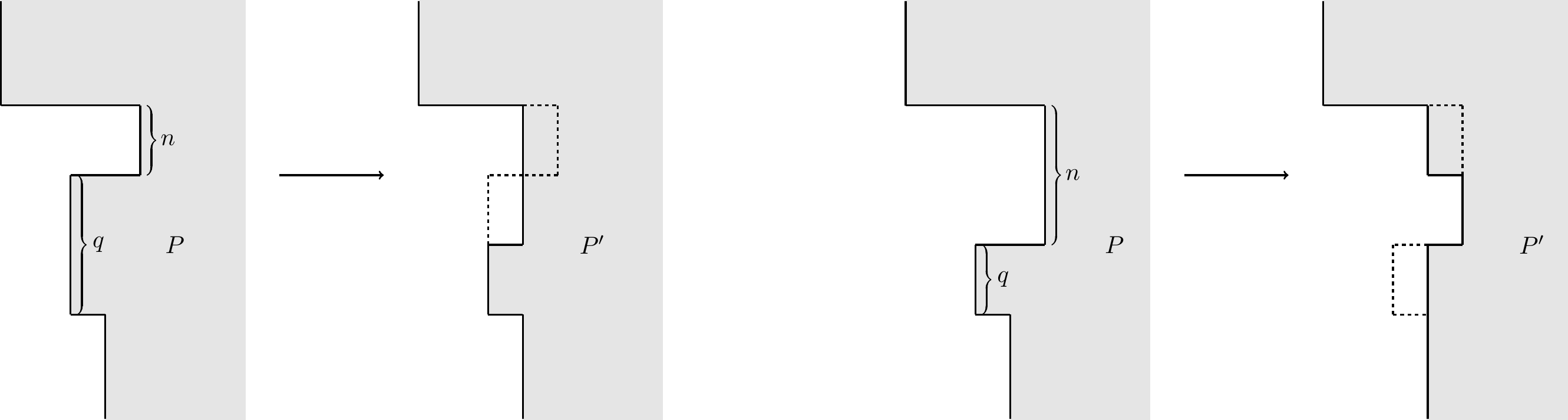}
    \caption{Two examples of polyominoes $P'$ obtained from $P$ by shifting $m$ unit squares, specifically, on the left $m=n$ and on the right $m=q$.}
 \label{fig:concavity_movement2}
\end{center}
\end{figure}
\end{proof}

\section{Identification of the maximal stability level}
\label{sec:maxstablevel}

In this section, we identify the maximal stability level of the energy landscape, which governs the low-temperature behavior of the Metropolis dynamics. Our approach combines carefully constructed monotone growth paths with the geometric insights established in the previous section, enabling both upper and lower bounds on the communication height between the stable configurations. The upper bound, derived in \cref{sub:upperbound}, is obtained by analyzing a family of explicitly defined paths along which opinion clusters expand through quasi-square and rectangular polyominoes, capturing the most energetically efficient ways to nucleate and grow competing phases. For the lower bound, presented in \cref{sub:lowerbound}, we leverage isoperimetric principles to characterize the mandatory “gates’’ that any optimal transition path must cross, identifying the critical configurations of minimal perimeter compatible with the required energy level. Together, these results delineate the full energetic structure of the transition between homogeneous states and provide a precise description of the critical droplets across all regimes of the parameter $\alpha$.

\subsection{Upper bound}
\label{sub:upperbound}
In this subsection, we describe specific paths instrumental for computing the upper bound on the energy barrier of the system. 
\begin{defi}
    We define the following paths:
    \begin{itemize}
        \item[(i)] The path \ $\bar\omega_1$ (resp. $\widetilde\omega_1$) starts from $\mmone$ (resp. $\ppone$)  and consists in the sequence of configurations that are increasing positive (resp.\ negative) clusters \emph{as close as possible to quasi-square shape}, where the last configuration $\eta_1^*$ contains a quasi-square $n \times (n+1)$ (resp.\ $m \times (m+1)$) of positive (resp. of negative) opinions inside $A$  (resp.\ $B$) in a sea of negative (resp.\ positive) opinions.
        \item[(ii)] The path $\bar\omega_2$ and $\widetilde\omega_2$ consist in the sequence of configurations starting from $\eta_1^*$ (defined in (i)) that are increasing rectangular clusters $$\Sigma_{\mathscr{R}^A_{n,n+1,1}}, \Sigma_{\mathscr{R}^A_{n,n+1,2}},...,\Sigma_{\mathscr{R}^A_{n,n+2,0}}, \Sigma_{\mathscr{R}^A_{n,n+2,1}},...,\Sigma_{\mathscr{R}^A_{n,N,0}}=\sigma_{\mathscr{A}_{0,0}}$$ 
        and
        $$\Sigma_{V\setminus\mathscr{R}^B_{m,m+1,1}}, \Sigma_{V\setminus\mathscr{R}^B_{m,m+1,2}},...,\Sigma_{V\setminus\mathscr{R}^B_{m,m+2,0}}, \Sigma_{V\setminus\mathscr{R}^B_{m,m+2,1}},...,\Sigma_{V\setminus\mathscr{R}^B_{m,N,0}}=\sigma_{\mathscr{A}_{k,k}},$$
        respectively.
        \item[(iii)] The path $\widetilde\omega_3$ starts from $\eta_1^*$ (defined in (i)) and consists in a sequence of configurations that are increasing clusters \emph{as close as possible to quasi-square shape}, whose last configuration $\eta_2^*$ contains a quasi-square $(2k+m) \times (2k+m+1)$ of negative opinions inside $B\cup S_1 \cup S_2$ in a sea of positive opinions. When the size of the cluster of negative opinions is at least $m(m+1)+1$, it is not entirely contained in $B$. 
        \item[(iv)] The path $\widetilde\omega_4$ starts from $\eta_2^*$ (defined in (iii)) and consists in the sequence of configurations starting from $\eta_2^*$ that are increasing rectangular clusters $$\Sigma_{V\setminus\mathscr{R}^{B\cup S_1 \cup S_2}_{2k+m,2k+m+1,1}}, \Sigma_{V\setminus\mathscr{R}^{B\cup S_1 \cup S_2}_{2k+m,2k+m+1,2}},...,\Sigma_{V\setminus\mathscr{R}^{B\cup S_1\cup S_2}_{2k+m,2k+m+2,0}}, \Sigma_{V\setminus\mathscr{R}^{B\cup S_1 \cup S_2}_{2k+m,2k+m+2,1}},...,\Sigma_{V\setminus\mathscr{R}^{B\cup S_1 \cup S_2}_{2k+m,N,0}}. $$
        \item[(v)] The path $\tilde \omega_5$ starts from $\eta_2^*$ (defined in (iii)) and consists in a sequence of configurations that are increasing clusters \emph{as close as possible to quasi-square shape}, whose last configuration is $\mmone$.
        \item[(vi)] The path $\tilde \omega_6$ connects $\sigma_{\mathscr{A}_{0,0}}$ and $\mmone$ as follows. Flip a plus (positive opinion) at distance one from a minus (negative opinion) in $S_1 \cup S_2$ and continue to flip the other positive opinions in the same column in such a way the perimeter of the cluster does not increase. Proceed in the same way for the adjacent columns up to $\mmone$.
        \item[(vii)] The path $\tilde \omega_7$ connects $\sigma_{\mathscr{A}_{k,k}}$ and $\mmone$ as follows. Flip a plus (positive opinion) at distance one from a minus (negative opinion) in $B$ and continue to flip the other positive opinions in the same column in such a way the perimeter of the cluster does not increase. Proceed in the same way for the adjacent columns up to $\sigma_{\mathscr{A}_{0,0}}$, and then follows $\tilde \omega_6$ to reach $\mmone$.
    \end{itemize}
\end{defi}

\begin{defi}
Assume $2\leq\alpha\leq n$. Let $\bar\omega_1^*$ (resp.\ $\bar\omega_2^*$) be the set of paths from $\ppone$ (resp.\ $\mmone$) to $\bigcup_{\ell,p=0}^k \{\sigma_{\mathscr{A}_{\ell,p}}\}$ defined as the composition of the paths $\widetilde\omega_1$ and $\widetilde\omega_2$ (resp.\ $\bar\omega_1$ and $\bar\omega_2$). Let  $\bar\omega_3^*$ be the set of paths defined as the composition of $\bar\omega_2^*$ and the \textit{time-reversal} of $\bar\omega_1^*$, where the time-reversal of a path is obtained by following the path backwards in time, so that the endpoint becomes the starting point, and the order of all intermediate configurations is reversed.
\end{defi}

\begin{defi}\label{def:cammini}
Assume $\alpha\geq m+1$.
\begin{itemize}
    \item[(i)] Let $\omega^*_1$ be the set of paths from $\ppone$ to $\mmone$ defined as the composition of $\omega'=(\omega_0,...,\omega_{mN})$ where
    \begin{align}
        \begin{cases}
            \omega_0 =\ppone \\
            \omega_i = \Sigma_{V\setminus\mathscr{C}^B_{0,i}} \qquad &\text{ if } i=1,...,N-1 \\
            \omega_{jN+i} = \Sigma_{V\setminus\mathscr{C}^B_{j,i}} \qquad &\text{ if } i=0,...,N \text{ and } j=1,...,m-1\\
            \omega_{mN} = \sigma_{\mathscr{A}_{k,k}}.
        \end{cases}
    \end{align}
    and $\tilde \omega_7$.
    \item[(ii)] Let $\omega^*_2$ be the set of paths from $\ppone$ to $\mmone$ defined as the compositions of the paths $\widetilde\omega_1$, $\widetilde\omega_2$ and $\tilde \omega_7$.
    \item[(iii)] Let $\omega^*_3$ be the set of paths from $\ppone$ to $\mmone$ defined as the compositions of the paths $\widetilde\omega_1$, $\widetilde\omega_3$, $\widetilde\omega_4$ and $\widetilde\omega_6$.
    \item[(iv)] Let $\omega^*_4$ be the set of paths from $\ppone$ to $\mmone$ defined as the compositions of the paths $\widetilde\omega_1$, $\widetilde\omega_3$ and $\tilde \omega_5$.
    \end{itemize}
    \end{defi}

Each of the paths introduced above corresponds to a distinct geometric mechanism for creating or annihilating large opinion clusters, and different paths become energetically optimal in different regimes of the interaction parameter $\alpha$.

In what follows, we compute the communication height along the peculiar paths described above. Each of the following propositions refers to different values of the parameter $\alpha$.

\begin{prop}[Communication height for $2\leq\alpha<n$]\label{prop:upperbound1}
Assume $2\leq\alpha<n$. Then we get
\begin{itemize}
\item[(i)] $
   \Phi\left(\ppone,\displaystyle\bigcup_{\ell,p=0}^k \{\sigma_{\mathscr{A}_{\ell,p}}\}\right) \leq \Phi_{\bar\omega_1^*}\left(\ppone,\displaystyle\bigcup_{\ell,p=0}^k \{\sigma_{\mathscr{A}_{\ell,p}}\}\right)
   = N (m-n) + 2\alpha^2-\alpha(N^2-2) -2$.
   %H(\sigma^B_{prot}) + 2(\alpha-1)$.
    \item[(ii)] $\Phi\left(\mmone,\displaystyle\bigcup_{\ell,p=0}^k \{\sigma_{\mathscr{A}_{\ell,p}}\}\right) \leq \Phi_{\bar\omega_2^*}\left(\mmone,\displaystyle\bigcup_{\ell,p=0}^k \{\sigma_{\mathscr{A}_{\ell,p}}\}\right)
    =N (m-n) + 2\alpha^2-\alpha(N^2-2) -2$.
    %H(\sigma^A_{prot}) + 2(\alpha-1)$.
\end{itemize}
\end{prop}

\begin{proof}
We start with case $(i)$. Recalling that $\bar{\omega}_1^*=(\tilde \omega_1, \tilde \omega_2)$, we compute the maximal energy along these two parts of the path. Considering $\tilde \omega_1$, the system evolves towards $\eta_1^*$ by crossing the configurations containing a square or a quasi-square of negative opinions in $B$. The energy of such configurations, for $\ell=1,...,m$ and $k=0,...,\ell-1$, is
\[
\begin{array}{lll}
   &H(\Sigma_{V\setminus\mathscr{Q}_{\ell,k}^B}) &= N (n-m) - \alpha N^2 + 2(Nm-\ell^2-k-Nn) +2\alpha(2\ell +\mathbf{1}_{ \{k \neq 0\}}),\notag \\
    &H(\Sigma_{V\setminus\mathscr{R}_{\ell,\ell-1,k}^B}) &= N (n-m) - \alpha N^2 + 2(Nm-\ell(\ell-1)-k-Nn) +2\alpha(2\ell-1 +\mathbf{1}_{ \{k \neq 0\}})\notag.
\end{array}
\]
By direct computation, we deduce that the maximum is attained in $\Sigma_{V\setminus\mathscr{R}_{\alpha,\alpha-1,1}^B}$.
Considering $\tilde \omega_2$, the system evolves towards $\bigcup_{\ell,p=0}^k \{\sigma_{\mathscr{A}_{\ell,p}}\}$ by growing rectangles of negative opinions in $B$. The energy of such configurations $\Sigma_{V\setminus\mathscr{R}_{m,a,k}^B}$, with $a=m+1,...,N-2$ and $k=0,...,m-1$, is
\begin{align}
  H(\Sigma_{V\setminus\mathscr{R}_{m,a,k}^B})=N (n-m) - \alpha N^2+2(Nm-ma-k-Nn)+2\alpha(a+m+ \mathbf{1}_{ \{k \neq 0\}})\notag.
\end{align}
By direct computation, we deduce that this function attains a maximum for $a=m+1$ and $k=1$. Thus, we finally need to compare $H(\Sigma_{V\setminus\mathscr{R}_{\alpha, \alpha-1,1}^B})$ with $H(\Sigma_{V\setminus\mathscr{R}_{m,m+1,1}^B})$ and we find that the maximum is attained in $\Sigma_{V\setminus\mathscr{R}_{\alpha, \alpha-1,1}^B}$.

The claim $(ii)$ follows after arguing as in case $(i)$.
\end{proof}

\begin{prop}[Communication height for $\alpha=n$]\label{prop:upperbound2}
Assume $\alpha=n$. Then the following statements hold.
\begin{itemize}
    \item[(i)] If $m=n$, then
    \[
    \Phi(\mmone, \ppone) \leq \Phi_{\bar\omega_3^*}(\mmone,\ppone) =
    2n^2-n(N^2-2) -2.
    %H(\sigma_{\mathscr{R}^A_{n,a,1}})= H(\sigma_{V\setminus\mathscr{R}_{n,a,1}^B})
    \]
    %for any $a=n-1,...,N-2$.
    \item[(ii)] If $m>n$, then
    \[
    \Phi\left(\mmone,\bigcup_{\ell,p=0}^k \{\sigma_{\mathscr{A}_{\ell,p}}\}\right) \leq \Phi_{\bar\omega_2^*}\left(\mmone,\bigcup_{\ell,p=0}^k \{\sigma_{\mathscr{A}_{\ell,p}}\}\right) 
    =N (m-n) + 2n^2-n(N^2-2) -2.
    %= H(\sigma_{\mathscr{R}^A_{n,a,1}})
    \]
    %for any $a=n-1,...,N-2$.
\end{itemize}
\end{prop}

 \begin{proof}
We start with the case $(i)$. Recalling that $\bar{\omega}_3^*$ is the composition of the paths $\bar{\omega}_2^*$ and the time reversal of $\bar{\omega}^*_1$, we compute the maximal energy along these two parts of the path. As we computed in the proof of Proposition \ref{prop:upperbound1}, the maximum along $\bar{\omega}_1^*$ is attained in $\Sigma_{\mathscr{R}^A_{\alpha,\alpha-1,1}}$ and the other one is attained in $\Sigma_{V\setminus\mathscr{R}^B_{\alpha, \alpha-1, 1}}$. 

The claim $(ii)$ follows after arguing as in case $(i)$.
\end{proof}

\begin{prop}[Communication height for $m+1\leq\alpha<2k + m$]\label{prop:upperbound3}
Assume $m+1\leq\alpha<2k + m$. Then 
\[
\Phi(\mmone,\ppone) \leq \min\{\Phi_{\omega_1^*}(\mmone,\ppone),\Phi_{\omega_2^*}(\mmone,\ppone),\Phi_{\omega_3^*}(\mmone,\ppone),\Phi_{\omega_4^*}(\mmone,\ppone)\}.
\]
\end{prop}

\begin{proof}
The statement follows directly from the definition of communication height.
\end{proof}

\begin{prop}[Communication height for $\alpha\geq 2k + m$]\label{prop:upperbound4}
Assume $\alpha\geq 2k + m$. Then the following statements hold.
\begin{itemize}
\item[(i)] If $2k + m \leq \alpha < \alpha^*$ and $m=n$, then 
\[
\Phi(\ppone,\mmone) \leq \Phi_{\omega_2^*}\left(\ppone,\mmone\right) 
= -N(n+m) -\alpha N^2 +2 (2m-1) +2\alpha(N+m-1).
%= H(\sigma_{\mathscr{R}^A_{m,N-2,1}})=H(\s_{V \setminus\mathscr{R}^{B}_{m,N-2,1}}).
\]
\item[(ii)] If $2k + m \leq \alpha < \alpha^*$ and $m>n$, then 
\[
\Phi(\ppone,\mmone) \leq \Phi_{\omega_2^*}\left(\ppone,\mmone\right) 
= -N(n+m) -\alpha N^2 +2 (2m-1) +2\alpha(N+m-1).
%=H(\s_{V \setminus\mathscr{R}^{B}_{m,N-2,1}}).
\]

\item[(iii)] If $\alpha \geq \alpha^*$ and $m=n$, then
\[
\Phi(\ppone,\mmone) \leq \Phi_{\omega_1^*}\left(\ppone,\mmone\right) =N(m-n)-\alpha N^2+2(\alpha-1)(N+1).
\]

\item[(iv)] If $\alpha \geq \alpha^*$ and $m>n$, then
  \[
\Phi(\ppone,\mmone) \leq \Phi_{\omega_1^*}\left(\ppone,\mmone\right) =N(m-n)-\alpha N^2+2(\alpha-1)(N+1).
\]
\end{itemize}
\end{prop}

\begin{proof}
We start with case $(i)$. Recall that $\omega_2^*=(\tilde \omega_1, \tilde \omega_2, \tilde \omega_7)$. We compute the maximal energy along these three parts of the path.
Considering $\tilde \omega_1$, the system evolves towards $\eta_1^*$ by crossing the configurations containing a square or a quasi-square of negative opinions in $B$. The energy of such configurations, for $\ell=1,...,m$ and $k=0,...,\ell-1$, is
\[
\begin{array}{lll}
   &H(\Sigma_{V\setminus\mathscr{Q}_{\ell,k}^B}) &= N (n-m) - \alpha N^2 + 2(Nm-\ell^2-k-Nn) +2\alpha(2\ell +\mathbf{1}_{ \{k \neq 0\}}),\notag \\
    &H(\Sigma_{V\setminus\mathscr{R}_{\ell,\ell-1,k}^B}) &= N (n-m) - \alpha N^2 + 2(Nm-\ell(\ell-1)-k-Nn) +2\alpha(2\ell-1 +\mathbf{1}_{ \{k \neq 0\}})\notag.
\end{array}
\]
By direct computation, we deduce that the maximum is attained in $\Sigma_{V\setminus\mathscr{R}_{m,m-1,1}^B}$. 

Considering $\tilde \omega_2$, the system evolves towards $\bigcup_{\ell,p=0}^k \{\sigma_{\mathscr{A}_{\ell,p}}\}$ by growing rectangles of negative opinions in $B$. The energy of such configurations $\Sigma_{V\setminus\mathscr{R}_{m,a,k}^B}$, with $a=m+1,...,N-2$ and $k=0,...,m-1$, is
\begin{align}
  H(\Sigma_{V\setminus\mathscr{R}_{m,a,k}^B})=N (n-m) - \alpha N^2+2(Nm-ma-k-Nn)+2\alpha(a+m+ \mathbf{1}_{ \{k \neq 0\}})\notag.
\end{align}
By direct computation, we deduce that this function attains a maximum for $a=N-2$ and $k=1$.

Considering $\tilde\omega_7$, by Lemmas \ref{lem:colonna1} and \ref{lem:enmin}$(i)$, the maximum along this path is
\begin{align}
    H(\Sigma_{\mathscr{C}_{1,1}^A}) &= N(m-n)-\alpha N^2+2(\alpha-1)(N+1)
\end{align}
Finally, by comparing these three values of the energy, we conclude.

The claim $(ii)$ follows after arguing as in case $(i)$.

Next, we analyze case $(iii)$. Recall that $\omega_1^*=(\omega', \tilde \omega_7)$. We compute the maximal energy along these two parts of the path. Considering $\omega'$, by Lemmas \ref{lem:colonna1} and \ref{lem:enmin}$(i)$, the maximum along this path is
\begin{align}
    H(\Sigma_{V\setminus\mathscr{C}_{1,1}^B}) &= N(m-n)-\alpha N^2+2(\alpha-1)(N+1).
\end{align}

Considering $\tilde\omega_7$, by Lemmas \ref{lem:colonna1} and \ref{lem:enmin}$(i)$, the maximum along this path is
\begin{align}
    H(\Sigma_{\mathscr{C}_{1,1}^A}) &= N(m-n)-\alpha N^2+2(\alpha-1)(N+1).
\end{align}

Finally, by noting that these two energy values coincide, we conclude.

The claim $(iv)$ follows after arguing as in case $(iii)$.
\end{proof}

The combination of Propositions \ref{prop:upperbound1}-\ref{prop:upperbound4} gives us an upper bound on the maximal stability level $\Gamma^*$ for any allowed value of $\alpha$.

\subsection{Lower bound}
\label{sub:lowerbound}
In this subsection, we provide a lower bound on the maximal stability level for any allowed value of $\alpha$. The approach we use relies on finding configurations with minimal energy, which are connected to those with minimal perimeter; namely, we need to solve an isoperimetric problem. Addressing it when $\alpha \geq m+1$ is particularly challenging due to the presence of neutral agents, which introduce heterogeneity in the space. Specifically, in the sets $S_1$ and $S_2$, the area takes on different values compared to other regions, while the perimeter remains constant throughout the space. For this reason, when identifying $\Gamma^*$ and the critical configurations for $\alpha \geq m+1$, we are only able to give a weaker result (see Conjecture \ref{conj:sushi} below).

We start by providing the two formal statements when $2\leq\alpha< n$ and $\alpha=n$, respectively.

\begin{lem}[Gate for the transition for $2\leq\alpha<n$]\label{lem:lower1}
     Suppose that $2\leq\alpha<n$, then:
    \begin{itemize}
        \item[(i)] any $\omega\in(\mmone\to \bigcup_{\ell,p=0}^k \{\sigma_{\mathscr{A}_{\ell,p}}\})_{opt}$ must pass through $\sigma^{A}\in \cG^A$; 
        \item[(ii)] any $\omega\in(\ppone\to \bigcup_{\ell,p=0}^k \{\sigma_{\mathscr{A}_{\ell,p}}\})_{opt}$ must pass through $\s^{B}\in \cG^B$. 
    \end{itemize}
\end{lem}

\begin{proof}[Proof of Lemma \ref{lem:lower1}]
We can assume that the path $\omega$ is \textit{non-backtracking}, i.e., it crosses only once each manifold $\cV_p$, with $0\leq p \leq N^2$.

\medskip
\noindent
{\bf Case (i).} 
Any path $\omega:\mmone\to\bigcup_{\ell,p=0}^k \{\sigma_{\mathscr{A}_{\ell,p}}\}$ can cross the manifold $\cV_{\alpha(\alpha-1)}$ in a configuration winding around the toric grid or not. By \eqref{eq:Ham}, it is immediate to see that the configurations with minimal energy are those having the maximal number of positive (negative, respectively) opinions in $A$ ($B$, respectively). We consider the following situations:
\begin{itemize}
   \item[(a)] $\omega$ crosses the manifold $\cV_{\alpha(\alpha-1)}$ in a configuration not winding around the toric grid.
   \item[(b)] $\omega$ crosses the manifold $\cV_{\alpha(\alpha-1)}$ in a configuration winding around the toric grid.
\end{itemize}
 
We start with case (a). By Proposition \ref{lem:isoperimetrica}(ii), we know that the unique (modulo translations and rotations) configuration of minimal perimeter is $\s_{prot}^{A}\in \cG^{A}_{prot}\cap\cV_{\alpha(\alpha-1)}$. In particular, $H(\s_{prot}^{A}) = \Gamma^* - 2\alpha + 2$. Any other configuration belonging to $\cV_{\alpha(\alpha-1)}$ not winding around the toric grid has energy at least $\Gamma^*+2$, since by \cite{Cerf2013} we know that its perimeter increases by at least 2. This would lead to a non-optimal path, which is not admissible. Thus, $\omega$ crosses the manifold $\cV_{\alpha(\alpha-1)}$ in the configuration $\s_{prot}^{A}$ and in order to increase the number of positive opinions from such a configuration, the unique admissible move is to swap a minus (negative opinion) located in a site of $A$ adjacent to the quasi-square of positive opinions. Indeed, for any other moves, we find a configuration with energy strictly greater than $\Gamma^*$, and this is not admissible for the optimality of the path. By \cite{Manzo2004} we know that this plus (positive opinion) has to be adjacent to the longest side of the quasi-square, otherwise the path would again not be optimal.

Consider now case (b). By Proposition \ref{lem:isoperimetrica}(i), we know that the unique (modulo translations and rotations) configurations of minimal perimeter, and hence minimal energy, is $\Sigma_{\mathscr{C}_{k,\alpha(\alpha-1)-kN}^A}\in\cV_{\alpha(\alpha-1)}$, where $k=\lfloor \frac{\alpha(\alpha-1)}{N} \rfloor$. In particular, $H(\Sigma_{\mathscr{C}_{k,\alpha(\alpha-1)-kN}^A}) = \Gamma^* -4\alpha^2+2\alpha(N+1)+2$, which exceeds $\Gamma^*$ because $\alpha<n<N/2$. This would lead to a non-optimal path, which is not admissible. This concludes the proof of case (i).

\medskip
\noindent
  {\bf Case (ii).} The proof uses the same argument as that of the lemma before, after replacing the manifold $\cV_{\alpha(\alpha-1)}$ with $\cV_{N^2-\alpha(\alpha-1)}$, which the configuration $\sigma^B_{prot}$ belongs to. This concludes the proof.
\end{proof}

\begin{lem}[Gate for the transition for $\alpha=n$]\label{lem:bitterballen}
Suppose that $\alpha=n$, then:
\begin{itemize}
    \item[(i)] if $m=n$, any $\omega\in(\mmone\to\ppone)_{opt}$ must pass through $\Sigma_{\mathscr{R}_{n,a,1}^A}$ and $\Sigma_{V\setminus\mathscr{R}_{n,a,1}^B}$ for any $a=n-1,...,N-2$;
    \item[(ii)] if $m>n$, any $\omega\in(\mmone\to\bigcup_{\ell,p=0}^k \{\sigma_{\mathscr{A}_{\ell,p}}\})_{opt}$ must pass through $\Sigma_{\mathscr{R}_{n,a,1}^A}$ for any $a=n-1,...,N-2$.
\end{itemize}
\end{lem}

\begin{proof}[Proof of Lemma \ref{lem:bitterballen}]
In all the cases, we can assume that the path $\omega$ is non-backtracking, i.e., it crosses only once each manifold $\cV_p$, with $0\leq p \leq N^2$.

{\bf Case (i).} Since $m=n$, the path $\omega:\mmone\to\ppone$ can be viewed as the composition of two optimal paths, say $\omega_1$ and $\omega_2$, such that $\omega_1:\mmone\to\bigcup_{\ell,p=0}^k \{\sigma_{\mathscr{A}_{\ell,p}}\}$, $\omega_2:\bigcup_{\ell,p=0}^k \{\sigma_{\mathscr{A}_{\ell,p}}\}\to\ppone$ and the time reversal of $\omega_2$ coincides with $\omega_1$ by symmetry after changing the role of spins $+$ and $-$. We can then focus on the path $\omega_1$ only. This path can cross the manifold $\cV_{\alpha(\alpha-1)}$ in a configuration winding around the toric grid or not. By \eqref{eq:Ham}, it is immediate to see that the configurations with minimal energy are those having the maximal number of positive (negative, respectively) opinions in $A$ ($B$, respectively). We consider the following situations:
\begin{itemize}
   \item[(a)] $\omega$ crosses the manifold $\cV_{\alpha(\alpha-1)}$ in a configuration not winding around the toric grid.
   \item[(b)] $\omega$ crosses the manifold $\cV_{\alpha(\alpha-1)}$ in a configuration winding around the toric grid.
\end{itemize}
 
We start with case (a). By Proposition \ref{lem:isoperimetrica}(ii) we know that the unique (modulo translations and rotations) configurations of minimal perimeter are $\s_{prot}^{A}\in \cG^{A}_{prot}\cap\cV_{\alpha(\alpha-1)}$. In particular, $H(\s_{prot}^{A}) = \Gamma^* - 2\alpha + 2$. Any other configuration belonging to $\cV_{\alpha(\alpha-1)}$ not winding around the toric grid has energy at least $\Gamma^*+2$, since by \cite{Cerf2013} we know that its perimeter increases by at least 2. This would lead to a non-optimal path, which is not admissible. Thus, $\omega_1$ crosses the manifold $\cV_{\alpha(\alpha-1)}$ in the configuration $\s_{prot}^{A}$ and in order to increase the number of positive opinions from such a configuration, the unique admissible move is to swap a minus (negative opinion) located in a site of $A$ adjacent to the quasi-square of positive opinions. Indeed, for any other moves, we find a configuration with energy strictly greater than $\Gamma^*$, and this is not admissible for the optimality of the path. By \cite{Manzo2004} we know that this plus (positive opinion) has to be adjacent to the longest side of the quasi-square, otherwise the path again would not be optimal, and the path crosses all the configurations $\Sigma_{\mathscr{R}_{n,a,1}^A}$, with $a=n,...,N-2$.

Consider now case (b). By Proposition \ref{lem:isoperimetrica}(i), we know that the unique (modulo translations and rotations) configurations of minimal perimeter, and hence minimal energy, is $\Sigma_{\mathscr{C}_{k,\alpha(\alpha-1)-kN}^A}\in\cV_{\alpha(\alpha-1)}$, where $k=\lfloor \frac{\alpha(\alpha-1)}{N} \rfloor$. In particular, $H(Sigma_{\mathscr{C}_{k,\alpha(\alpha-1)-kN}^A}) = \Gamma^* -4\alpha^2+2\alpha(N+1)+2$, which exceeds $\Gamma^*$ because $n<N/2$. This would lead to a non-optimal path, which is not admissible. 

{\bf Case (ii).} We can argue as in case (i).
\end{proof}

In what follows, we deal with the case $\alpha\geq m+1$. The conjecture below states that any path other than $\omega_1^*,\omega_2^*,\omega_3^*,\omega_4^*$ (defined in Definition \ref{def:cammini})is not optimal. This, in turn, characterizes the saddle points along these paths by determining the critical configurations. This makes it a useful tool for estimating the lower bound of the energy barrier. More precisely, the configurations that we identify as critical under this conjecture are those containing square, rectangular, or column-shaped clusters. Specifically, these are configurations that include clusters associated with polyominoes of minimal perimeter. The difference compared to other cases lies in the fact that these clusters can belong to two or more distinct zones (for example, $A$ and $S_1$), rather than to a single one, as shown in Figure \ref{fig:conj}.

\begin{figure}
\begin{center}
    \includegraphics[scale=0.38]{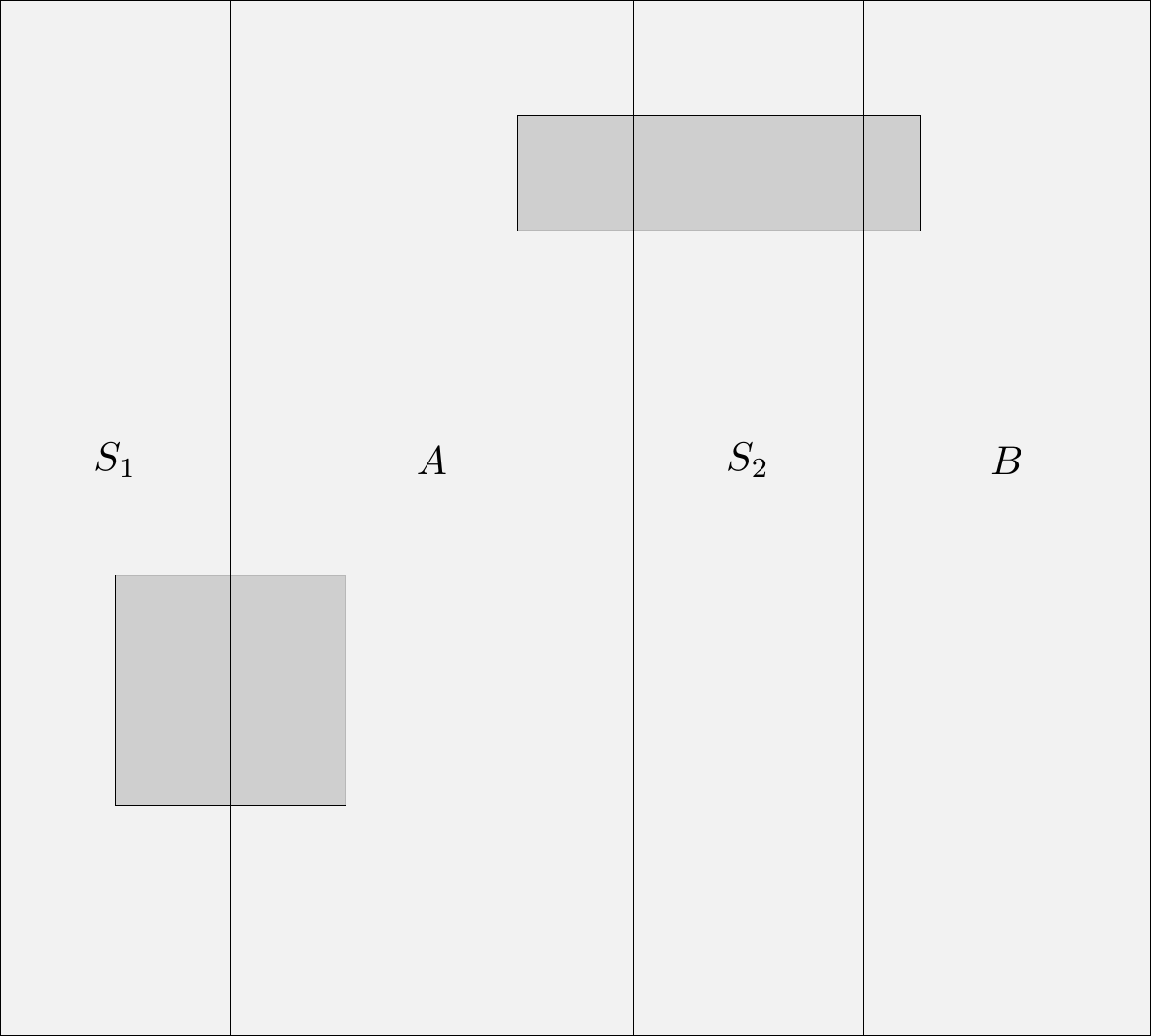}
    \caption{Examples of a configuration with square and rectangular clusters of positive opinions belonging to $S_1\cup A_1$ and $A\cup S_2 \cup B$, respectively.}
 \label{fig:conj}
\end{center}
    \end{figure}

\begin{conj}\label{conj:sushi}
If $\alpha \geq m+1$, any path $\omega\not \in \{ \omega^*_{1}, \omega^*_{2}, \omega^*_{3}, \omega^*_{4}\}$ which connects $\ppone$ to $\mmone$ is such that either $\Phi_{\omega} > \min\{\Phi_{\omega^*_{1}},\Phi_{\omega^*_{2}},\Phi_{\omega^*_{3}}, \Phi_{\omega^*_{4}}\}$, or $\Phi_{\omega} = \min\{\Phi_{\omega^*_{1}},\Phi_{\omega^*_{2}},\Phi_{\omega^*_{3}}, \Phi_{\omega^*_{4}}\}$ and $\mathcal{S}(\omega)\subseteq\{\mathcal{S}(\omega^*_1),\mathcal{S}(\omega^*_2),\mathcal{S}(\omega^*_3), \mathcal{S}(\omega^*_4)\}$.
\end{conj}

\begin{rem}\label{rmk:computations}
Recall Definition \ref{def:cammini} for the definition of the paths $\omega_1^*,\omega_2^*,\omega_3^*,\omega_4^*$. To understand what is the value of $\min\{\Phi_{\omega^*_{1}},\Phi_{\omega^*_{2}},\Phi_{\omega^*_{3}}, \Phi_{\omega^*_{4}}\}$, the idea is to develop direct computations to find the minimum. By the definition of the paths, we have
\begin{equation}\label{eq:panzanella}
\begin{array}{lll}
\Phi_{\omega_{1}^*}&=\max\{\Phi_{\omega'},\Phi_{\tilde \omega_{7}}\}, \qquad \qquad \quad \ \ \ \Phi_{\omega_{2}^*}&=\max\{\Phi_{\tilde \omega_{1}},\Phi_{\tilde \omega_{2}},\Phi_{\tilde \omega_{7}}\}, \\
\Phi_{\omega_{3}^*}&=\max\{\Phi_{\tilde \omega_{1}},\Phi_{\tilde \omega_{3}},\Phi_{\tilde \omega_{4}},\Phi_{\tilde \omega_{6}}\}, \qquad
\Phi_{\omega_{4}^*}&=\max\{\Phi_{\tilde \omega_{1}},\Phi_{\tilde \omega_{3}},\Phi_{\tilde \omega_{5}}\}.
\end{array}
\end{equation}
For $g(m,n)=\max \{m/2, n+1 \}$, we obtain 
    \footnotesize
    \begin{align}
        &\Phi_{\omega'}-H(\ppone)=2(Nm-N-1-Nn)+2\alpha(N+1)  + 2N (n-m), \notag\\
        &\Phi_{\tilde \omega_{1}}-H(\ppone)=
        \begin{cases}
            2(Nm-\alpha(\alpha+1)-1-Nn)+2\alpha(2\alpha+2)+2N (n-m) &\text{ if } \alpha \leq m, \notag \\
            2(Nm-m(m+1)-1-Nn)+2\alpha(2m+2)+2N (n-m) &\text{ if } \alpha >m ,
        \end{cases}  \\
        &\Phi_{\tilde \omega_{2}}-H(\ppone)=2(Nm-m(N-2)-1-Nn)+2\alpha(m+N-1)+2N (n-m), \notag \\
        &\Phi_{\tilde \omega_{3}}-H(\ppone)=
        \begin{cases}
            2(Nm-(2k+m)(2k+m+1)-1-Nn)+2\alpha(4k+2m+2)+2N (n-m) &\text{ if } \alpha \geq g(m,n), \notag \\
            2(Nm-m(m+1)-1-Nn)+2\alpha(2m+2)+2N (n-m) &\text{ if } \alpha < g(m,n)  ,
        \end{cases}  \\ 
        &\Phi_{\tilde \omega_{4}}-H(\ppone)=2(Nm-(2k+m)(N-2)-1-Nn)+2\alpha(2k+m+N-1)+2N (n-m), \notag\\
        &\Phi_{\tilde \omega_{5}}-H(\ppone)=
            2(Nm-(N-2)(N-2)-1-Nn+(n-2)(N-2))+2\alpha(2N-3)+2N (n-m), \notag \\
            &\Phi_{\tilde\omega_{6}}-H(\ppone)=2(Nn-N-1-Nm)+2\alpha(N+1)  + 2N (m-n), \notag \\
         &\Phi_{\tilde\omega_{7}}-H(\ppone)=2(Nn-N-1-Nm)+2\alpha(N+1)  + 2N (m-n). \notag
    \end{align}
    \normalsize
Since the computations are quite long and technical, and not inspiring in terms of reasoning, in what follows we only explicitly report what happens in the case $\alpha \geq 2k+m$. By using \eqref{eq:panzanella}, we deduce that 
\[
\Phi_{\omega_{1}^*}=\Phi_{\omega'}, \qquad
\Phi_{\omega_{2}^*}=\Phi_{\tilde \omega_{2}}, \qquad
\Phi_{\omega_{3}^*}=\Phi_{\tilde \omega_{4}}, \qquad
\Phi_{\omega_{4}^*}=\Phi_{\tilde \omega_{5}}.
\]
Finally, recalling that $\alpha^*$ is defined in \eqref{eq:alpha*},
\[
\Phi_{\omega_1^*}=\min \{ \Phi_{\omega_{1}^*}, \Phi_{\omega_{2}^*},\Phi_{\omega_{3}^*}, \Phi_{\omega_{4}^*}\} \hbox{ if and only if } \alpha>\alpha^*,
\]
otherwise,
\[
\Phi_{\omega_{2}^*}=\min \{ \Phi_{\omega_{1}^*}, \Phi_{\omega_{2}^*},\Phi_{\omega_{3}^*}, \Phi_{\omega_{4}^*}\}.
\]
\end{rem}
In the following, we present the results we derive using Conjecture \ref{conj:sushi}.
\begin{lem}[Gate for the transition for $m+1 \leq\alpha < 2k+m$]\label{lem:lower3}
    Suppose $m+1 \leq\alpha < 2k+m$, then any path $\omega\in(\ppone\to\mmone)_{opt}$ must pass through $\mathcal{S}(\omega)\subseteq\{\mathcal{S}(\omega^*_1),\mathcal{S}(\omega^*_2),\mathcal{S}(\omega^*_3), \mathcal{S}(\omega^*_4)\}$.
\end{lem}
\begin{proof}[Proof of Lemma \ref{lem:lower3}]
The statement follows directly from applying Conjecture \ref{conj:sushi}.
\end{proof}

\begin{lem}[Gate for the transition for $2k+m \leq \alpha <\alpha^*$]\label{lem:lower4}
    Suppose $2k+m \leq \alpha <\alpha^*$, then the following statements hold.
    \begin{itemize}
    \item[(i)] If $m=n$, then any $\omega\in(\ppone\to\mmone)_{opt}$ must pass through $\Sigma_{\mathscr{R}^A_{m,N-2,1}}$ and $\Sigma_{V \setminus\mathscr{R}^{B}_{m,N-2,1}}$.
    \item[(ii)] If $m>n$, then any $\omega\in(\ppone\to\mmone)_{opt}$ must pass through $\Sigma_{V\setminus\mathscr{R}^{B}_{m,N-2,1}}$.
    \end{itemize}
\end{lem}
\begin{proof}[Proof of Lemma \ref{lem:lower4}]
By applying Conjecture \ref{conj:sushi}, we know that any $\omega\in(\ppone\to\mmone)_{opt}$ is such that
$\Phi_{\omega} = \min\{\Phi_{\omega^*_{1}},\Phi_{\omega^*_{2}},\Phi_{\omega^*_{3}}, \Phi_{\omega^*_{4}}\}$ and $\mathcal{S}(\omega)\subseteq\{\mathcal{S}(\omega^*_1),\mathcal{S}(\omega^*_2),\mathcal{S}(\omega^*_3), \mathcal{S}(\omega^*_4)\}$. Since $2k+m\leq \alpha<\alpha^*$, by Remark \ref{rmk:computations} we deduce that 
\begin{align}
    \Phi_{\omega_{2}^*}=\min \{ \Phi_{\omega_{1}^*}, \Phi_{\omega_{2}^*},\Phi_{\omega_{3}^*}, \Phi_{\omega_{4}^*}\}=
        -N(n+m) -\alpha N^2 +2(2m-1) +2\alpha(N+m-1).
\end{align}
and along $\omega_{2}^*$ the unique configurations with this energy are $\Sigma_{\mathscr{R}^A_{m,N-2,1}}$ and $\Sigma_{V\setminus\mathscr{R}^B_{m,N-2,1}}$ if $m=n$, otherwise, if $m>n$, it is $\Sigma_{V\setminus\mathscr{R}^B_{m,N-2,1}}$.
\end{proof}

\begin{lem}[Gate for the transition for $\alpha>\alpha^*$]\label{lem:lower5}
     Suppose that $\alpha>\alpha^*$, then the following statements hold.
    \begin{itemize}
    \item[(i)] If $m=n$, then any $\omega\in(\ppone\to\mmone)_{opt}$ must pass through $\Sigma_{\mathscr{C}_{1,1}^A}$ and $\Sigma_{V\setminus\mathscr{C}_{1,1}^B}$. 
    \item[(ii)] If $m>n$, then any $\omega\in(\ppone\to\mmone)_{opt}$ must pass through $\Sigma_{V\setminus\mathscr{C}_{1,1}^B}$.
    \end{itemize}
\end{lem}

\begin{proof}[Proof of Lemma \ref{lem:lower5}]
By applying Conjecture \ref{conj:sushi}, we know that any $\omega\in(\ppone\to\mmone)_{opt}$ is such that
$\Phi_{\omega} = \min\{\Phi_{\omega^*_{1}},\Phi_{\omega^*_{2}},\Phi_{\omega^*_{3}}, \Phi_{\omega^*_{4}}\}$ and $\mathcal{S}(\omega)\subseteq\{\mathcal{S}(\omega^*_1),\mathcal{S}(\omega^*_2),\mathcal{S}(\omega^*_3), \mathcal{S}(\omega^*_4)\}$. Since $\alpha>\alpha^*$, by Remark \ref{rmk:computations} we deduce that 
\begin{align}
    \Phi_{\omega_{1}^*}=\min \{ \Phi_{\omega_{1}^*}, \Phi_{\omega_{2}^*},\Phi_{\omega_{3}^*}, \Phi_{\omega_{4}^*}\}=
        N(m-n) -\alpha N^2 + 2(N+1)(\alpha-1).
\end{align}
and along $\omega_{1}^*$ the unique configurations with this energy are $\Sigma_{\mathscr{C}_{1,1}^A}$ and $\Sigma_{V\setminus\mathscr{C}_{1,1}^B}$ if $m=n$, otherwise, if $m>n$, it is $\Sigma_{V\setminus\mathscr{C}_{1,1}^B}$.
\end{proof}

\section{Recurrence property}
\label{sec:recurrence}
Having identified the structure of the energy landscape and its critical configurations, we now analyze the resulting consequences for the opinion dynamics. This section establishes the recurrence properties of the Metropolis dynamics in the low-temperature regime by precisely identifying the set of configurations whose stability level does not exceed the critical threshold $L^*=2(\alpha-1)$. Building on the geometric characterization of cluster shapes developed earlier, the analysis proceeds by partitioning the non-stable configurations into a finite number of classes according to the structure of their positive and negative opinion clusters across the regions $A,B,$ and $S$. 
For each class, we construct explicit downhill moves that either strictly decrease the energy or do so after a finite sequence of local modifications, thereby showing that these configurations cannot lie at or above the maximal stability level.

The main result of this section is stated below, in Proposition \ref{teoRP}, which proves that only the stable and metastable configurations can exhibit stability level $L^*$ and that the probability of remaining outside this set for times exceeding is super-exponentially small. The resulting recurrence property is a key ingredient for the full metastability analysis, ensuring that the dynamics rapidly returns to the relevant portion of the energy landscape from which rare transitions occur.

\begin{prop}[Recurrence property]\label{teoRP} 
If $L^*=2(\alpha-1)$, then \mbox{$\mathcal{X}_{L^*}={\cal X}^{stab} \cup {\cal X}^{meta}$} and for any $\epsilon>0$, the function
\begin{align}\label{recurrence2J}
    \beta \mapsto \sup_{\sigma \in \mathcal{X}} \mathbb{P}_{\sigma}(\tau_{\mathcal{X}_{L^*}}> {\rm e}^{\beta(L^*+\epsilon)})
\end{align}
is $\SES(\beta)$, where we write $\SES(\beta)$ for any function of $\beta$ that decays to zero faster than any exponential of $\beta$.
\end{prop}

We partition the set of configurations $\cX \setminus ({\cal X}^{stab} \cup {\cal X}^{meta})$ as follows:
\begin{itemize}
\item $\mathscr{X}_1$ is the set of configurations containing at least a cluster $C^+$ of positive opinions in $A$ with at least an angle of $\frac{3}{4}\pi$.
\item $\mathscr{X}_2$ is the set of configurations containing in $A$ only convex clusters of positive opinions.
\item $\mathscr{X}_3$ is the set of configurations containing no cluster of positive opinions in $A$ and at least a cluster $C^-$ of negative opinions in $B$ with at least an angle of $\frac{3}{4}\pi$.
\item $\mathscr{X}_4$ is the set of configurations containing no cluster of positive opinions in $A$ and in $B$ only convex clusters of negative opinions.
\item $\mathscr{X}_5$ is the set of configurations containing no cluster of positive opinions in $A$ and no cluster of negative opinions in $B$, and at least a cluster $C^+$ of positive opinions in $S$ with at least an angle of $\frac{3}{4}\pi$.
\item $\mathscr{X}_6$ is the set of configurations containing no cluster of positive opinions in $A$ and no cluster of negative opinions in $B$, and only convex clusters of positive opinions in $S$.
\item $\mathscr{X}_7$ is the set containing the configuration with all minuses in $A$, all positive opinions in $B$, and all negative opinions in $S$.
\end{itemize}

\begin{figure}%[!htb]
        \begin{center}
        \includegraphics[scale=0.4]{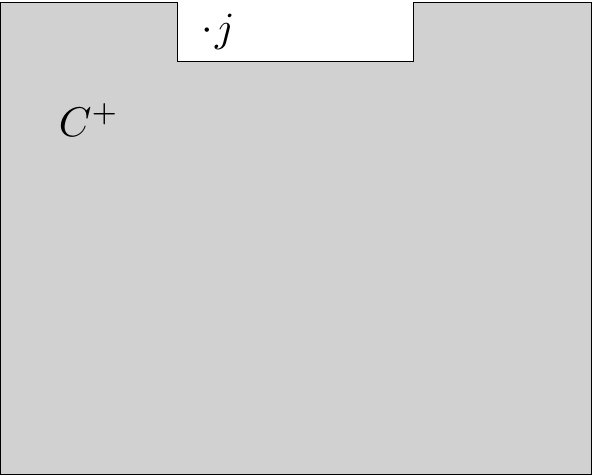} \,\,\,\,\,\,\,\,\,\,\,\, \includegraphics[scale=0.4]{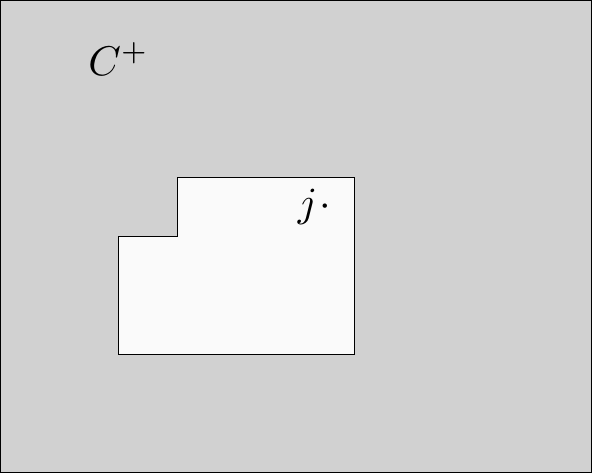}  \,\,\,\,\,\,\,\,\,\,\,\, \includegraphics[scale=0.4]{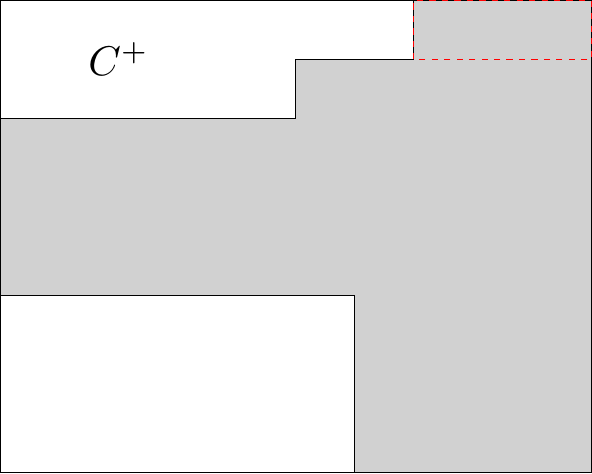}
        \end{center}
        \caption{Examples of configurations appear in the proof of Lemmas~\ref{noX} and \ref{noX5}, where the clusters of positive opinions are represented in grey.}
        \label{fig:clusters}
    \end{figure}

\begin{lem}[$\mathscr X_1 \cup \mathscr X_3 \not \subset \mathcal{X}_0$] \label{noX}
A configuration $\sigma$ in $\mathscr X_1 \cup \mathscr X_3$ is not in $\mathcal{X}_0$.
\end{lem}

\begin{proof}[Proof of Lemma~\ref{noX}] 
Consider a configuration $\sigma \in\cX_1$. Suppose that $C^+$ has an internal angle of $\frac{3}{4}\pi$.
Let $j$ be the site at distance one to a site in $C_+$ such that $\sigma(j)=-1$ and that belongs to the unit square intersecting the boundary of $C^+$ in two or more edges, see the left panel in Figure \ref{fig:clusters}. Recall the definition of the configuration $\sigma^{(j)}$ given in \eqref{eq:spinconf}. Then,
\begin{equation*}
    H(\sigma^{(j)})-H(\sigma)\leq -2<0,
\end{equation*}
with $\sigma^{(j)}$ belongs to $\mathcal{I}_{\sigma}$. 
Thus, the stability level is equal to $\mathscr V_{\sigma}=0$ and it follows that $\sigma \not \in \mathcal{X}_0$.

The case for a configuration in $\cX_3$ follows by similar arguments.
\end{proof}

\begin{lem}[$\mathscr X_5 \not \subset \mathcal{X}_0$]\label{noX5}
A configuration $\sigma$ in $\mathscr X_5$ is not in $\mathcal{X}_0$.
\end{lem}

\begin{proof}[Proof of Lemma~\ref{noX5}] 
Consider a configuration $\sigma \in \cX_5$. We distinguish two cases:
\begin{itemize}
    \item[(i)] the cluster $C^+$ contains at least one hole.
    \item[(ii)] the cluster $C^+$ contains no hole.
\end{itemize}

We start with case (i). Let us fix a hole of cardinality $k$, with $k\geq1$. Let $j$ be a site at distance one to a site in $C^+$ such that $\sigma(j)=-1$ and that belongs to a unit square in the hole intersecting the boundary of $C^+$ in two or more edges, see the central panel of Figure \ref{fig:clusters}, then
\begin{equation*}
    H(\sigma^{(j)})-H(\sigma)\leq 0.
\end{equation*}
If the equality holds, we iterate the procedure until we get a strict inequality. This procedure terminates in at most $k$ steps because the hole's cardinality is finite, and at the last step, the perimeter of the cluster decreases. Thus, the stability level is equal to $\mathscr V_{\sigma}=0$ and it follows that $\sigma \not \in \mathcal{X}_0$.

Next, we analyze case (ii). By the assumption on $\sigma$, there exists a strip of positive opinions in the boundary $C^+$ intersecting the boundary of the circumscribing rectangle, see the right panel of Figure \ref{fig:clusters}. We flip all pluses (positive opinions) in this strip and we obtain a configuration $\tilde \sigma$
\begin{equation*}
    H(\tilde\sigma)-H(\sigma) = -2\alpha <0.
\end{equation*}
Thus the stability level is equal to $\mathscr V_{\sigma}=0$ and it follows that $\sigma \not \in \mathcal{X}_0$.
This concludes the proof.
\end{proof}

\begin{lem}[$\mathscr X_2 \cup \mathscr X_4 \not \subset \mathcal{X}_{2 (\alpha-1)}$]\label{XPIPPO}
A configuration $\sigma$ in $\mathscr X_2 \cup \mathscr X_4$ is not in $\mathcal{X}_{2 (\alpha-1)}$.
\end{lem}

\begin{proof}
Consider a configuration $\sigma \in\mathscr X_2$. We distinguish three cases:
\begin{itemize}
    \item[(i)] the configuration contains at least a cluster $C^+$ with rectangular shape.
    \item[(ii)] the configuration contains only clusters with column shape, but $\sigma_{|A}\neq\ppone_{|A}$.
    \item[(iii)] $\sigma_{|A}=\ppone_{|A}$.
\end{itemize}

We start with case (i). Suppose first that the cluster $C^+$ is totally contained in $A$. Let $\ell$ be the horizontal side length of $C^+$. If $\ell < \alpha$, then we flip all pluses (positive opinions) along this side and we obtain a configuration $\tilde \sigma$ such that
\begin{equation*}
    H(\tilde\sigma)-H(\sigma) =2(\ell-\alpha) < 0.
\end{equation*}
Suppose $\ell\geq\alpha$, then we flip all minuses (negative opinions) attached to this side and we obtain a configuration $\tilde \sigma$ such that
\begin{equation*}
    H(\tilde\sigma)-H(\sigma) = 2(-\ell+\alpha) \leq 0.
\end{equation*}
If the equality holds, then we iterate this procedure until the cluster winds around the toric grid, thereby leading to a smaller perimeter.

Suppose now that $C^+$ intersects $S$ or $ S \cup B$. Then, we flip all pluses (positive opinions) along the vertical side in $S$ or $ S \cup B$ and we obtain a configuration $\tilde \sigma$ such that
\begin{equation*}
    H(\tilde\sigma)-H(\sigma) \leq -2\alpha < 0.
\end{equation*}

We analyze case (ii). By assumption, there is at least a minus (negative opinion) in $A$. If $A$ contains at least a column of positive opinions, we flip all minuses (negative opinions) in $A$ starting from the minuses (negative opinions) nearest to a column of positive opinions in $A$ and call this configuration $\overline{\sigma}$. The first step has an energy cost equal to $2\alpha-2$ because the perimeter increases by $\alpha$ and the number of negative opinions decreases by $1$. The energy cost of the other $N-2$ steps is equal to $-2$ because the perimeter remains the same. After the last step, the energy decreases by $-2\alpha-2$. Thus, 
\begin{align}
    H(\overline{\sigma})-H(\sigma)=-2N<0
\end{align}
and $\mathscr V_\sigma=2(\alpha-1)$. 

Next, we consider case (iii). If there is at least a cluster $C^-$ of negative opinions in $B$ (resp.\ $C^+$ of positive opinions in S) with a $\frac{3}{4}\pi$ angle, then we proceed as in Lemma \ref{noX} (resp.\ Lemma \ref{noX5}). In the other cases, there is at least a cluster $C^-$ in $B$ (resp. $C^+$ in $S$) with rectangular or column shape. In the case of rectangular shapes in $B$ and $S$, the claim follows after arguing similarly to case (i). 

We are now left to consider the case in which only column shapes are present in $S\cup B$. 
Consider first the case in which $\alpha\leq n$ or $\alpha>n$ and $n=m$. In this case there is at least a column of positive opinions in $B$, otherwise $\sigma_{|B}=\mmone_{|B}$ and $\sigma\in\bigcup_{\ell,p=0}^k \{\sigma_{\mathscr{A}_{\ell,p}}\}$, which is a contradiction. 
We flip all pluses (positive opinions) in $B$ starting from the nearest to the column of negative opinions and we call this configuration $\sigma'$. Thus, the first step has an energy cost of $2\alpha-2$, since the perimeter increases by $\alpha$ and the number of positive opinions decreases by $1$. The energy cost of the other $N-2$ steps is equal to $-2$ because the perimeter remains the same. After the last step the energy decreases by $-2\alpha-2$. Thus, 
\begin{align}
    H(\sigma')-H(\sigma)=-2N<0
\end{align}
and $\mathscr V_\sigma\leq2(\alpha-1)$.

Assume now $\alpha>n$ and $n<m$. By the initial assumption, we have that $\alpha>m$.
Suppose first that there is at least a column of negative opinions in $S$, then we flip the column of negative opinions nearest to the set $A$ as above. 
Thus, we are left with the case in which $\sigma_{|S}=\ppone_{|S}$. In this case, there is at least one column of negative opinions in $B$, otherwise $\sigma=\ppone$, which is a contradiction.
Assume first that there is also a column of positive opinions in $B$. In this case, we flip all the pluses (positive opinions) in the column nearest to the column of negative opinions as above.
Otherwise, $\sigma_{|B}=\mmone_{|B}$ and therefore $\sigma\in\bigcup_{\ell,p=0}^k \{\sigma_{\mathscr{A}_{\ell,p}}\}$. 
We note that $\sigma$ belongs to the optimal path $\omega_1^*$ and $\omega_2^*$. Thus, we reduce $\sigma$ with an energy cost equal to $V^*<\Gamma^*$ following one of these two paths.

Consider a configuration $\sigma \in\mathscr X_4$. We distinguish two cases:
\begin{itemize}
    \item[(i)] the configuration contains only positive opinions adjacent to the set $A$. 
    \item[(ii)] the configuration contains at least a negative opinion adjacent to the set $A$.
    \end{itemize}
Recall that $\sigma_{|A}=\mmone_{|A}$. 

In case (i), we consider the configuration $\bar\sigma$ obtained from $\sigma$ after flipping all minuses (negative opinions) belonging to one of the two columns of $A$ adjacent to the set $S$. The first move has an energy cost of $2\alpha-2$, and the following ones have an energy cost of $-2$, except the last one, which costs $-2\alpha-2$. Thus, $H(\bar\sigma) = H(\sigma) - 2N < H(\sigma)$ and $\mathscr V_{\sigma}\leq 2\alpha-2$.

Finally, in case (ii), we consider the nearest column to the set $A$ containing a positive opinion. This column exists, since $\sigma\neq\mmone$. We then consider the configuration $\bar\sigma$ obtained from $\sigma$ after flipping all positive opinions adjacent to the set $A$ starting from the plus (positive opinion) nearest to a minus (negative opinion) in $S$. All these moves have an energy cost smaller than or equal to $0$, except the last one, which costs $-2\alpha$ at most. Thus, $H(\bar\sigma) \leq H(\sigma) - 2\alpha < H(\sigma)$ and $\mathscr V_{\sigma}=0$. This concludes the proof.
\end{proof}

\begin{lem}[$\mathscr X_6 \not \subset \mathcal{X}_{2 (\alpha-1)}$]\label{noX6}
A configuration $\sigma$ in $\mathscr X_6$ is not in $\mathcal{X}_{2(\alpha-1)}$.
\end{lem}
\begin{proof}
We distinguish two cases:
\begin{itemize}
\item[(i)] the configuration contains in $S$ at least a cluster $C^+$ with rectangular shape,
\item[(ii)] the configuration contains in $S$ only clusters with column shape.
\end{itemize}
First, consider case (i). We flip all pluses (positive opinions) along the vertical side in $S$ nearest to $A$ and we obtain a configuration $\tilde \sigma$ such that
\begin{equation*}
    H(\tilde\sigma)-H(\sigma) \leq -2\alpha < 0.
\end{equation*}

Finally, consider case (ii). We note that $S$ contains at least a column of positive opinions, otherwise $\sigma \in \mathscr{X}_7\setminus\mathscr{X}_6$. Thus, we flip all minuses (negative opinions) in $S$ starting from the minuses (negative opinions) nearest to a column of positive opinions in $S$ and call this configuration $\overline{\sigma}$. The first step has an energy cost equal to $2\alpha$ because the perimeter increases by $\alpha$. The energy cost of the other $N-2$ steps is equal to $0$ because the perimeter remains the same. After the last step the energy decreases by $-2\alpha$. Thus, 
\begin{align}
    H(\overline{\sigma})=H(\sigma)
\end{align}
and we iterate this procedure until the resulting configuration is in $\mathscr{X}_7$. Thus, we can apply Lemma \ref{noX7} and find $\mathscr V_{\sigma}\leq \max \{2, \mathscr V_{\mathscr{X}_7}\}$.

\end{proof}

\begin{lem}[$\mathscr X_7 \not \subset \mathcal{X}_{2 (\alpha-1)}$]\label{noX7}
A configuration $\sigma$ in $\mathscr X_7$ is not in $\mathcal{X}_{2(\alpha-1)}$.
\end{lem}

\begin{proof}
    We observe that if $\sigma\in \mathscr{X}_7$, then $\sigma_{|A}=\mmone_{|A}$, $\sigma_{|B}=\ppone_{|B}$ and $\sigma_{|S}=\mmone_{|S}$.
    In this case, we flip all pluses (positive opinions) in $B$ starting from one of the two columns adjacent to $S$ and call this configuration $\overline{\sigma}$. The first step has an energy cost equal to $2\alpha-2$ because the perimeter increases by $\alpha$ and the number of positive opinions decreases by $1$. The energy cost of the other $N-2$ steps is equal to $-2$ because the perimeter remains the same. After the last step, the energy decreases by $-2\alpha-2$. Thus, 
\begin{align}
    H(\overline{\sigma})-H(\sigma)=-2N<0
\end{align}
and $\mathscr V_\sigma\leq2(\alpha-1)$.
\end{proof}

\section{Proofs of the main results}
\label{sec:proofs}
This section collects the proofs of the main theorems outlined in~\cref{sec:defmodel}, relying on the bounds and the recurrence property derived in the previous two sections.

\subsection{Proof of Theorem \ref{thm:min}}
\begin{proof}
We need to find the configurations $\s\in\cal X$ that minimize the Hamiltonian in \eqref{eq:Ham}. Since $|\gamma(\s)|\geq0$ for any configuration $\s$, we start by considering configurations $\s$ such that $|\gamma(\s)|=0$, i.e., $\sigma\in\{\mmone,\ppone\}$. Then we have
\[
H(\mmone)= N (n-m) -\alpha N^2
\]
and
\[
H(\ppone)= N (m-n) - \alpha N^2.
\]
Since $n\leq m$, we have $H(\mmone)\leq H(\ppone)$ and the equality holds if and only if $n=m$.

Suppose now $|\gamma(\s)|>0$ and consider the restriction of the configuration $\s$ to the set $S_1\cup S_2$, i.e. $\s_{|S_1\cup S_2}$. It is immediate to see that the spins of the configuration in this set cannot decrease the energy due to the assumption $s_i=0$ if $i\in S_1\cup S_2$. For this reason, we start by assuming that $|\gamma(\s_{|S_1\cup S_2})|=0$, so that we have to find
\[
\min_{\s\in\cal X} \Big ( H(\mmone) + 2 (M_B(\s)-M_A(\s))  + \alpha |\gamma(\s_{|A\cup B})| \Big )
= H(\mmone) + \min_{\s\in\cal X} \Big (2 (M_B(\s)-M_A(\s))  + \alpha |\gamma(\s_{|A\cup B})| \Big ).
\]
By \cite[Corollary 6.18]{Cirillo2013}, we deduce that $|\gamma(\s_{|A\cup B})|\geq 4\sqrt{M_A(\s)+M_B(\s)}$. Thus, we need to find
\[
\min_{\substack{0\leq M_A(\s)\leq|A| \\ 0\leq M_B(\s)\leq|B|}} \Big (2 (M_B(\s)-M_A(\s))  + 4 \alpha \sqrt{M_A(\s)+M_B(\s)} \Big ).
\]
It is immediate to check that the minimum of this function is attained at the boundary of the domain, i.e., for $M_A(\s)\in\{0,nN\}$ and $M_B(\s)\in\{0,mN\}$. After selecting the values for $M_A(\s)$ and $M_B(\s)$, we can provide a better estimate of $|\gamma(\s_{|A\cup B})|$ thanks to the periodic boundary conditions we consider. In the cases $M_A(\s)=M_B(\s)=0$ and $M_A(\s)=nN$, $M_B(\s)=mN$, the minimizing configurations are $\sigma\equiv\mmone$ and $\sigma\equiv\ppone$, respectively. This case has already been treated. In the case $M_A(\s)=0$ and $M_B(\s)=nB$, the minimal energy is
\[
H(\mmone) + 2mN + 2\alpha N > H(\mmone),
\]
so that this configuration cannot minimize the energy on $\cal X$. In the case $M_A(\s)=nN$ and $M_B(\s)=0$, the minimal energy is
\[
H(\mmone) - 2nN + 2\alpha N,
\]
which is less than $H(\mmone)$ if and only if $\alpha< n$, and is attained by the configuration $H(\sigma_{\mathscr{A}_{\ell,p}})$ for $\ell,p\in\{0,k\}$. To identify all the other stable configurations in the case $\alpha< n$, we use Lemma \ref{lem:energiatruth}.
\end{proof}

\subsection{Proof of Theorem \ref{thm:meta}}
Combining Lemmas \ref{lem:lower1}-\ref{lem:lower5} and Propositions \ref{prop:upperbound1}-\ref{prop:upperbound4}, we identify the communication height $\Phi(\sigma,\eta)$, where
\begin{align}
    (\sigma,\eta) =
    \begin{cases}
        (\mmone,\bigcup_{\ell,p=0}^k \{\sigma_{\mathscr{A}_{\ell,p}}\}) \text{ and } (\ppone, \bigcup_{\ell,p=0}^k \{\sigma_{\mathscr{A}_{\ell,p}}\}) &\text{ if } 2\leq \alpha < n, \\
        (\mmone,\ppone) &\text{ if } \alpha = n \text{ and }  n = m, \\
         (\mmone,\bigcup_{\ell,p=0}^k \{\sigma_{\mathscr{A}_{\ell,p}}\}) &\text{ if } \alpha = n \text{ and }  n < m, \\
         (\ppone,\mmone) &\text{ if } \alpha \geq m + 1.
    \end{cases}
\end{align}

Moreover, in Proposition \ref{teoRP}, we find that this value is the maximal stability level: any other configuration different from $\sigma$ has a lower stability level. Thus, we can conclude by applying \cite[Theorem 2.4]{Cirillo2013}, indeed the first assumption of \cite[Theorem 2.4]{Cirillo2013} is satisfied for the choice of $A =\sigma$ and $a=\Gamma^*$, and the second assumption of \cite[Theorem 2.4]{Cirillo2013} is satisfied thanks to Proposition \ref{teoRP}.

\subsection{Proof of Theorem \ref{thm:transitiontime}}

We consider separately the cases $\sigma\in{\cal X}^{meta}$ and $\sigma\in{\cal X}^{stab}$, i.e., the asymptotic transition and tunneling time, respectively.

We first treat the case $(\sigma,\eta)\in{\cal X}^{meta}\times{\cal X}^{stab}$. Thus, Theorem \ref{thm:transitiontime} follows from \cite[Theorems 4.1, 4.9 and 4.15]{Manzo2004} together with Theorems \ref{thm:min},\ref{thm:meta}. Theorem \ref{thm:transitiontime}(iv) follows from \cite[Proposition 3.24]{NZB15} with $\widetilde\Gamma(\cX\setminus\{s\})=\Gamma^*$ for any $s\in{\cal X}^{stab}$.

Consider now the case $(\sigma,\eta)\in{\cal X}^{stab}\times{\cal X}^{stab}$, i.e., we are considering the tunneling time between two stable configurations. Thanks to \cite[Lemma 3.6]{NZB15}, we deduce that for our model the quantity $\widetilde\Gamma(B)$, with $B\subsetneq\cX$, defined in \cite[eq. (21)]{NZB15} is such that $\widetilde\Gamma(\cX\setminus\{\eta\})=\Gamma^*$. Moreover, thanks to the property of absence of deep cycles in Proposition \ref{teoRP}, \cite[Proposition 3.18]{NZB15} implies that $\Theta(\sigma,\eta) = \Gamma^*$ for $\sigma,\eta\in{\cal X}^{stab}$. Thus, Theorem \ref{thm:transitiontime}(i) follows from \cite[Corollary 3.16]{NZB15}. Moreover, Theorem \ref{thm:transitiontime}(ii) follows from \cite[Theorem 3.17]{NZB15} provided that \cite[Assumption A]{NZB15} is satisfied: this is implied by the absence of deep cycles and \cite[Proposition 3.18]{NZB15}. Finally, Theorem \ref{thm:transitiontime}(iii) follows from \cite[Theorem 3.19]{NZB15} provided that \cite[Assumption B]{NZB15} is satisfied: this is implied by the absence of deep cycles and the argument carried out in \cite[Example 4]{NZB15}. Theorem \ref{thm:transitiontime}(iv) follows from \cite[Proposition 3.24]{NZB15} with $\widetilde\Gamma(\cX\setminus\{\eta\})=\Gamma^*$ for any $\eta\in{\cal X}^{stab}$.

\subsection{Proof of Theorem \ref{thm:selle}}

The claim directly follows after using Lemmas \ref{lem:lower1}-\ref{lem:lower5} for different values of $\alpha$.

\section{Conclusions and future work}
\label{sec:conclusions}
We introduced an Ising-type model for opinion dynamics that separates immutable hidden preferences from publicly expressed binary opinions and incorporates neutral agents as an active structural component. In a low-temperature regime, this separation leads to a rich metastable energy landscape whose structure depends sensitively on the spatial organization of hidden preferences. Using the pathwise approach, we provided a complete characterization of stable and metastable configurations and derived sharp asymptotics for transition and mixing times.

Several directions for future research naturally arise. A first challenge would be to relax the strong symmetry assumptions on the underlying graph. Extending the analysis to less regular networks would clarify which aspects of the metastable behavior persist beyond the toric grid setting and which are geometry-specific.

A second direction concerns the structure of hidden preferences. While we focused on deterministic and spatially organized patterns to isolate geometric effects, it would be of interest to study random or disordered hidden preferences, possibly with spatial correlations. In this respect, it would be very interesting to understand how quenched disorder interacts with neutrality and metastability.

From a modeling standpoint, allowing hidden preferences to evolve over longer time scales or interact dynamically with public opinion could capture processes of persuasion, reinforcement, or opinion fatigue. Such multiscale extensions may give rise to new forms of metastable behavior not captured by the present framework.

Finally, the geometric techniques developed here, particularly the isoperimetric inequalities for polyominoes winding around the torus, have a much broader applicability. Adapting these tools to other lattice spin systems or interacting particle models with periodic boundary conditions may provide new avenues for the rigorous analysis of metastability in other complex systems.

\printbibliography

\end{document}